\documentclass[reqno,11pt]{amsart}
\usepackage{amscd,amsfonts,amsmath,amssymb,amstext,amsthm}
\usepackage[utf8]{inputenc}
\usepackage[T2A]{fontenc}
\usepackage[russian,english]{babel}
\usepackage{cite}
\usepackage[unicode]{hyperref}
\usepackage{color}
\usepackage{xcolor}
\usepackage{comment}
\usepackage{tikz}
\usepackage{tikz-cd}
\usepackage[normalem]{ulem}

\usetikzlibrary{matrix,arrows,decorations.pathmorphing}

\makeatletter
\def\@settitle{\begin{center}%
    \baselineskip14\p@\relax
    \bfseries
    \MakeUppercase{\@title}
  \end{center}%
}
\makeatother

\newtheorem{theorem}{Theorem}[section]
\newtheorem{assertion}{Assertion}
\newtheorem{lemma}{Lemma}
\numberwithin{lemma}{section}
\newtheorem{proposition}{Proposition}[section]
\newtheorem{corollary}{Corollary}[section]
\theoremstyle{remark}

\newtheorem{remark}{Remark}[section]
\newtheorem{example}{Example}[section]
\newtheorem{definition}{Definition}[section]
\numberwithin{definition}{section}
\numberwithin{equation}{section}

\def\A{{\mathbb A}}
\def\C{{\mathbb C}}

\def\Cn*{{\C^n}^*}
\def\CN*{{\C^N}^*}
\def\T{{\mathbb T}}
\def\L{{\mathbb L}}
\def\V{{\mathbb V}}

\def\rank{{\rm rank\:}}
\def\R{{\mathbb R}}
\def\Z{{\mathbb Z}}
\def\codim{{\rm codim\:}}
\def\codima{{\rm codim_a}}
\def\dima{{\rm dim_a}}

\def\vol{{\rm vol}}

\def\rf{{\rm RF}}
\def\E{\:{\rm e}}
\def\re{{\rm Re}\:}
\def\im{{\rm Im\:}}
\def\supp{{\rm supp\:}}

\def\EA{{\rm E}}
\def\nwt{{\rm nwt}}

\begin{document}
\title{
On the exponential algebraic geometry
}
\author{Boris Kazarnovskii}
\address {\noindent Higher School of Modern Mathematics MIPT
\newline
1 Klimentovskiy per., Moscow, Russia
\newline
{\it kazbori@gmail.com}.}
\begin{abstract}
The set of roots of any finite system of exponential sums in the space $\mathbb{C}^n$ is called an exponential variety.
We define
the intersection index of varieties of complementary dimensions,
and the ring of classes of numerical equivalence of exponential varieties
with operations "addition-union" and "multiplication-intersection".
This ring is analogous to the ring of conditions of the torus $(\mathbb{C}\setminus 0)^n$
and is called the ring of conditions of 
$\mathbb{C}^n$.
We provide its description in terms of convex geometry.
Namely we associate an exponential variety with
an element of a certain ring generated by convex polyhedra in $\mathbb{C}^n$.
We call this element the Newtonization of the exponential variety.
For example, the Newtonization of an exponential hypersurface is its Newton polyhedron.
The Newtoniza\-tion map defines an isomorphism of the ring of conditions to
the ring generated by convex polyhedra in $\mathbb{C}^n$.
It follows, in particular,
that the intersection index of $n$ exponential hypersurfaces is equal
to the mixed pseudo-volume of their Newton polyhedra.
\end{abstract}
\maketitle
\tableofcontents
\section{Introduction.}
\label{intro}
\subsection{Exponential Sums}\label{intro1}
Let $\Lambda \subset \Cn*$ be a finite set of linear functionals in the complex vector space $\mathbb{C}^n$.
Recall that a function on $\mathbb{C}^n$ of the form:
$$f(z)=\sum_{\lambda\in\Lambda,\:0\ne c_\lambda\in\C\setminus0}c_\lambda \E^{\lambda(z)},$$
is called an exponential sum.
The finite set $\supp(f) = \Lambda$ and it's convex hull $\Gamma(f) = \text{conv}(\Lambda)$ are respectively called the support and the Newton polytope of the exponential sum $f$.
Exponential sums form a ring.

By external signs, the ring of exponential sums resembles the ring of Laurent polynomials.
For example, these rings are generated respectively by characters of the additive group $\mathbb{C}_+^n$ of the space $\mathbb{C}^n$ and by characters of the complex torus $(\mathbb{C}\setminus 0)^n$.
However, finding algebraic-geometric properties of the ring of exponential sums is a difficult task.
The first result in this direction was obtained in \cite{Ritt}: \emph{if the quotient of two exponential sums in one variable is an entire function, then this function is also an exponential sum.}
Note that there is no division with a remainder
in the ring of exponential sums of one variable.
Accordingly, Ritt's theorem is not obvious.
For many variables, the statement is proven in \cite{AG}.
An important distinction between exponential sums and polynomials is the infinity of the set of zeros of any non-invertible exponential sum of one variable.
For example, the set of zeros of the exponential sum $\mathrm{e}^{2\pi i z} - 1$ is $\mathbb{Z}$.
Therefore, the existence of a common zero of two exponential sums does not imply the existence of a common divisor.
Indeed, exponential sums $\E^z - 1$ and $\E^{2\pi i z} - 1$, which share the common zero at $z = 0$,
are coprime because otherwise,
the set of common zeros would be infinite.
The сonjecture about the finiteness of the set
of common zeros of two coprime exponential sums, possibly coming from J. Ritt, is currently far from being proven.

If one of the exponential sums is $\E^{2\pi iz}-1$,
then the Ritt's assumption
is a statement about the set of integer solutions of an exponential equation.
In this case, the assumption follows from the statement known as the ``Mordell-Lang conjecture'' for the complex torus \cite{L};
see also \cite{Ev}.
\subsection{Exponential Varieties}\label{intro2}
The zero set of a \emph{finitely generated ideal} of the ring of exponential sums is called an \emph{exponential variety} (abbreviated as \EA-variety).
The first small step in algebraic geometry of \EA-varieties was made in \cite{K97},
where the definition of the \emph{algebraic dimension} $\dima X$ of an \EA-variety $X$ was introduced.
Accordingly, $\codima X = n - \dima X$ is called the algebraic codimension of $X$.
The algebraic dimension may turn out to be negative.
For example, the algebraic dimension of an \EA-variety in $\mathbb{C}^1$
defined by the equations $\mathrm{e}^{iz} - 1 = \mathrm{e}^{2\pi i z} - 1 = 0$ is $-1$.
The algebraic dimension $\dima X$ does not exceed the dimension of $X$ as a closed analytic subset of $\mathbb{C}^n$.
An irreducible component $(X,z)$ of the germ of the \EA-variety $X$ at the point $z$
is called \emph{atypical} if its dimension as an analytic set in $\mathbb{C}^n$ is greater than $\dima X$.
\begin{example}\label{exAnomal}
Let a non-empty \EA-variety $X$ be defined by the equations $f = g = 0$.
Then, if the exponential sums $f$ and $g$ have no common divisor in the ring of exponential sums,
then $\codima X = 2$; otherwise $\codima X = 1$.
In particular, the point $z=0$ is an atypical component of the \EA-variety
defined by the equations $\mathrm{e}^{iz} - 1 = \mathrm{e}^{2\pi i z} - 1 = 0$ in $\mathbb{C}^1$.
In \cite{K97}, it is proven that any atypical component of an \EA-variety of algebraic codimension two is an affine subspace.
In \cite{K97} the assumption was also made about the geometry of atypical components of \EA-varieties of arbitrary codimension.
The assumption was proven by B. Zilber in \cite{Z02};
see Theorem \ref{thmAnom} in \textsection\ref{intro5}.
\end{example}
The concept of irreducibility for \EA-varieties is not defined.
A partial replace\-ment is the concept of \emph{equidimensionality}.
Recall that an algebraic variety is sometimes called equidimensional if the dimensions of its irreducible components are the same.
The concept of equidimensionality for \EA-varieties is defined in \textsection\ref{intro5}.
Every \EA-variety is a union of equidimensional \EA-varieties of different algebraic dimensions.
The set of equidimensional components of an \EA-variety is uniquely determined.
\begin{theorem}\label{thRes1}
Let $k\leq n$ and

$\bullet$
$\Lambda_1, \Lambda_2, \ldots, \Lambda_k$ are finite subsets of the exponent space $\Cn*$

$\bullet$
$S(\Lambda_1, \Lambda_2, \ldots, \Lambda_k)$ is the space of systems of $k$ exponential sums $f_1,\ldots,f_k$ such that $\supp(f_i)\subset\Lambda_i$

$\bullet$
$Y(f_1, f_2, \ldots, f_k)$ is an \EA-variety defined by the equations $f_1=\ldots=f_k=0$.

Then there exists the \EA-hypersurface $\mathfrak D$
in $S(\Lambda_1, \Lambda_2, \ldots, \Lambda_k)$ (that is the \EA-variety defined by a single equation)
such that
$$\forall (f_1, f_2, \ldots, f_k)\notin \mathfrak D\colon\:\codima(Y(f_1,\ldots,f_k))=k $$
\end{theorem}
The equation of the \EA-hypersurface $\mathfrak D$ can be seen as the exponential analog of the resultant of a system of $k$ polynomial equations.
The proof of the theorem is provided at the end of \textsection\ref{intro5}.

Below, we 1) define the intersection index of exponential varieties, and 2) construct the ring of classes of numerical equivalence of exponential cycles in the space $\C^n$
with operations of "addition-union" and "multiplication-intersection".
This ring is analogous to the ring of conditions of the complex torus $(\mathbb{C}\setminus0)^n$ or, more generally,
to the ring of conditions of a spherical homogeneous space.
It is called the ring of conditions of the space $\mathbb{C}^n$.
We establish the ring isomorphism $\mathcal{C}(\mathbb{C}^n) \to \mathfrak{S}{\mathfrak{p}}(\mathbb{C}^n)$,
where $\mathfrak{S}{\mathfrak{p}}(\mathbb{C}^n)$ is the ring defined in \textsection\ref{polRing},
generated by convex polytopes located in the space $\mathbb{C}^n$.

The similar ring isomorphism for algebraic varieties in $(\mathbb{C}\setminus0)^n$ is constructed in \cite{EKK}:
the ring of conditions of the torus $(\mathbb{C}\setminus0)^n$ is isomorphic to the subring in the ring of convex bodies in $\R^n$ generated by the Newton polytopes of algebraic hypersurfaces.
The image of a variety in the ring of convex bodies is called the Newtonization of the algebraic variety; see \textsection\ref{AlgNewt}.
Similarly, the image of an \EA-variety under the isomorphism $\mathcal{C}(\mathbb{C}^n)\to\mathfrak{S}_{\mathfrak{p}}(\mathbb{C}^n)$
is called the Newtonization of the \EA-variety; see Definition \ref{dfNewtonizationE}.
The Newtonization of an exponential hypersurface is its Newton polytope; see Definitions \ref{dfNewtonizationETtrop} and \ref{dfNewtonizationE}.

A functional $\lambda\in \mathbb{C}^n{}^*$ is called real if its values on the space $\re \C^n$ are real.
An exponential sum with real exponents  is called quasi-algebraic.
Quasi-algebraic exponential sums also form a ring.
If the ideal of equations of an \EA-variety is generated by quasi-algebraic exponential sums,
then the \EA-variety is also called quasi-algebraic.
Classes of equivalence containing linear combinations of quasi-algebraic \EA-varieties form a subring
of the ring $\mathcal{C}(\mathbb{C}^n)$ called the \emph{quasi-algebraic ring of conditions} of the space $\mathbb{C}^n$.
In the quasi-algebraic case, the geometry of \EA-varieties is simpler than in the general case.
In particular, in this case, the structure of the ring of conditions is close to the structure of the ring of conditions of the complex torus:
the quasi-algebraic ring of conditions is isomorphic to the ring $\mathfrak{S}_\mathrm{vol}(\mathbb{R}^n)$
generated by convex polytopes in the space $\mathbb{R}^n$; see \textsection\ref{polRing}.
(Recall that the ring of conditions of the torus $(\mathbb{C}\setminus0)^n$ is generated by polytopes in $\mathbb{R}^n$ with vertices at lattice points, see \cite{EKK}.)
The quasi-algebraic ring of conditions is constructed in \cite{K22}.
A detailed announcement of the construction was previously published in \cite[Chapter 6]{EKK}.
\begin{remark}
The ring of convex polytopes in $\mathbb{R}^n$ was defined in \cite{MCM}.
Subsequent studies of such rings were conducted by various authors;
 see \cite{MCM1,MCM2,KhP,Br2}.
 The 
 connection between the ring of convex polytopes and the ring of conditions of the torus $(\mathbb{C}\setminus0)^n$ turned out to be fruitful  for both algebraic and Euclidean geometry; see \cite{MCM2,FS,Br1,Est,EKK,K03,KhA}.
 Here, we use the definition of the ring of polytopes from \cite{EKK};
 see Definition \ref{dfGor}.
\end{remark}
\subsection{The Intersection Index of Exponential Varieties}\label{intro3}
In algebraic geometry, the intersection index typically arises as the number of points of a zero-dimensional algebraic variety.
However,
the set of points of a zero-dimensional \EA-variety is infinite.
We define \emph{the weak density  $d_w(X)$ of zero-dimensional \EA-variety $X$} as an analog of the number of points of zero-dimensional algebraic variety;
see Definition \ref{dfWeakDens}.

Using the concept of the weak density,
will define \emph{the real-valued intersection index} $I(X,Y)$
for \EA-varieties $X$ and $Y$ such that $\dima X + \dima Y = n$;
see Definition \ref{dfIndex}.
\begin{theorem}\label{thmInd}
For equidimensional \EA-varieties $X$ and $Y$ such that $\dima X + \dima Y = n$,
there exists an exponential hypersurface $Z\subset\C^n$ such that,
for all $z\not\in Z$, the following is true:
1) $\dima \left((z+X)\cap Y\right) = 0$, and 2) $d_w((z+X)\cap Y)=I(X,Y)$.
 \end{theorem}
\begin{theorem}\label{thmBKK}
Let $X_1,\ldots,X_n$ be exponential hypersurfaces defined by the equations $f_1=\ldots=f_n=0$,
where $f_i$ is an exponential sum with Newton polytope $\Delta_i$.
Then the intersection index $I(X_1,\ldots,X_n)$ of the exponential hypersurfaces $X_i$ is equal
to the mixed pseudovolume of the Newton polytopes
$\Delta_i$ of exponential sums $f_i$, multiplied by $n!$;
see Definition \ref{dfPseudo} in \textsection\ref{intro4}.
\end{theorem}
The proof of the theorem is given at the end of \textsection\ref{dw}.
Previous versions of Theorem \ref{thmBKK} can be found in \cite{Dokl,olga,K84,Few}.

In the first part of the article, consisting of \textsection\ref{intro5}-\ref{dw},
all the concepts listed above are defined,
and a construction of the ring of conditions of the space $\mathbb{C}^n$ based on the concept of intersection index of \EA-varieties is provided.
However, the concept of intersection index from Definition \ref{dfIndex} weakly corresponds to the familiar notions of intersection index of algebraic varieties.
Properties of the intersection index, restoring the traditional context,
are presented in the second part of the article, consisting of \textsection\ref{index}-\ref{pr1}.
In this part, the theory of toric varieties is essential for the exposition.
\subsection{Ring of Conditions of $\C^n$}\label{intro4}
Let $\mathcal C(M)$ be the set of classes of numerical equivalence of varieties in a homogeneous algebraic variety $M$.
Sometimes on the set $\mathcal C(M)$, a ring structure can be correctly defined with the operations of "addition-union" and "multiplication-intersection".
Then the ring $\mathcal C(M)$ is called the \emph{ring of conditions for a homogeneous space} $M$.
This ring is an analog of the Chow ring of an algebraic variety; see \cite{FS}.
If the variety $M$ is not compact,
then compared to the Chow ring, the ring of conditions typically contains significantly more information about the intersections of algebraic cycles.
The ring of conditions was first constructed in the works \cite{CP,C} for symmetric spaces.
Subsequently, this construction was extended to arbitrary spherical varieties of linear reductive groups.
If the homogeneous space $M$ of a linear reductive group is not spherical,
then the algorithm for constructing the ring of conditions $\mathcal{C}(M)$ from \cite{CP,C} is not applicable; see \cite{A1}.
Below, we consider \emph{exponential cycles}, i.e., linear combinations of \EA-varieties in the space $\mathbb{C}^n$ with real coefficients,
and construct the corresponding ring of conditions $\mathcal{C}(\mathbb{C}^n)$.
It is worth noting that the same ring $\mathcal{C}(\mathbb{C}^n)$
can be constructed using only integer coefficients in the mentioned linear combinations;
see property \hyperlink{R2}{($\bf R_2$)} below.
\begin{definition}\label{dfNumerEqu}
We will call equidimensional \EA-varieties $X, Y$ of algebraic codimension $k$ \emph{numerically equivalent} if $I(X,Z)=I(Y,Z)$ for any \EA-variety $Z$ of algebraic dimension $k$.
All \EA-varieties of algebraic codimension $>n$ are also considered numerically equivalent.
\end{definition}
The definition of the ring of conditions is based on the following statement.
\begin{theorem}\label{thmBase}
For any equidimensional \EA-varieties $X, Y$, there exist classes of numerical equivalence $\Pi(X, Y)$, $\Sigma(X, Y)$ such
that for some \EA-hypersurface $\mathfrak D(X, Y)$ in $\C^n$ depending on $X, Y$, the following is true.
If $z\in\C^n\setminus\mathfrak D(X, Y)$, then

\textbf{(i)} all \EA-varieties $X\cap(z+Y)$ are equidimensional and belong to the class $\Pi(X, Y)$;

\textbf{(ii)} if $\codima X=\codima Y$, then all \EA-varieties $X\cup(z+Y)$ are equi\-dimensional and belong to the class $\Sigma(X, Y)$;

\textbf{(iii)} the classes $\Pi(X, Y)$, $\Sigma(X, Y)$ depend only on the classes of numerical equivalence containing $X, Y$.
\end{theorem}
The proof of the theorem is provided in \textsection\ref{dw}.
For equidimensional \EA-variety $X$, we denote it's equivalence class as $\kappa(X)$.
We define $\kappa(X) + \kappa(Y) = \Sigma(X,Y)$ and $\kappa(X) \cdot \kappa(Y) = \Pi(X,Y)$.
It follows from Theorem \ref{thmBase} that the classes of numerical equivalence form a commutative associative semiring,
graded by algebraic codimensions of \EA-varieties.
\begin{definition}\label{dfRingCondTrop}
We denote by $\mathcal C_k$ the real vector space consisting of linear combinations of classes of numerical equivalence of equidimensional \EA-varieties with algebraic codimension $k$.
The commutative graded algebra $\mathcal C(\C^n)=\mathcal C_0+\ldots+\mathcal C_n$ is called the ring of conditions for the space $\C^n$.
\end{definition}
The equivalence classes containing linear combinations of quasialgebraic \EA-varieties form a subring of the ring $\mathcal C(\C^n)$ called the \emph{quasialgebraic ring of conditions for the space $\C^n$}.
The quasi-algebraic ring of conditions is isomorphic to the ring $\mathfrak{S}_\text{vol}(\mathbb{R}^n)$,
generated by convex polytopes in the space $\mathbb{R}^n$; see \textsection\ref{polRing}.
The construction of the quasi-algebraic ring of conditions was published in \cite{K22}.
A detailed overview of it can be found in \cite[Part 6]{EKK}.

Here are some of the main properties of the ring of conditions $\mathcal C(\C^n)$.
\begin{enumerate}
\item[\hypertarget{R1}{\textbf{(R1)}}] If the Newton polytopes of exponential sums $f$ and $g$ differ by a parallel shift in the space ${\C^n}^*$,
then the \EA-hypersurfaces $X$ and $Y$ defined by the equations $f=0$ and $g=0$ are numerically equivalent.
The converse is not true; see Example \ref{exPs1} (1).

\item[\hypertarget{R2}{\textbf{(R2)}}] Any element of the ring $\mathcal C(\C^n)$ belongs to a $\Z$-algebra generated by the classes of numerical equivalence of exponential hypersurfaces.

\item[\hypertarget{R3}{\textbf{(R3)}}] The spaces $\mathcal C_0$ and $\mathcal C_n$ are one-dimensional.
If $0$-dimensional \EA-varieties $X$ and $Y$ are numerically equivalent, then $d_w(X)=d_w(Y)$.
The weak density $d_w$ is a coordinate in the one-dimensional space $\mathcal C_n$.
If $k \neq 0$ and $k \neq n$, then the vector space $\mathcal C_k$ is infinite-dimensional.

\item[\hypertarget{R4}{\textbf{(R4)}}]
For $p+q=n$, the multiplication operation defines a non-degenerate pairing of real vector spaces $\mathcal C_p\times\mathcal C_q\to\mathcal C_n\xrightarrow{d_w}\R$.

\item[\hypertarget{R5}{\textbf{(R5)}}] Let the sum of algebraic codimensions of \EA-varieties $X_1, \ldots, X_k$ equals $n$,
and let $\bar X_i\in\mathcal C(\C^n)$ be the class of numerical equivalence of the \EA-variety $X_i$.
We call the intersection index of \EA-vartieties $X_i$
$$
I(X_1,\ldots,X_k)=d_w(\bar X_1\cdot\ldots\cdot\bar X_k),
$$
where $\bar X_1\cdot\ldots\cdot\bar X_k\in\mathcal C_n$ is a product of  equivalence classes $\bar X_i$ of $X_i$
in the ring $\mathcal C_n$.
If $X_i$ is a zero hypersurface of the exponential sum $f_i$ with a Newton polytope $\Delta_i$, then
$$
I(X_1,\ldots,X_n)=n!\:\mathfrak p(\Delta_1,\ldots,\Delta_n),
$$
where $\mathfrak p(\Delta_1,\ldots,\Delta_n)$ is the mixed pseudovolume of convex polytopes defined below (see definitions \ref{dfPseudo} and \ref{dfPseudoMixed}).
\end{enumerate}
\begin{definition}\label{dfPseudo}
The pseudovolume $\mathfrak p(\Delta)$ of a polytope $\Delta\subset{\C^n}^*$ is defined as
$$
\frac{1}{(2\pi)^n}\sum_{\Lambda\subset\Delta,\:\dim(\Lambda)=n}c(\Lambda)\:A(\Lambda)\:\vol_n(\Lambda),
$$
where the summation is over all $n$-dimensional facets $\Lambda$ of $\Delta$.
Here,

(i) $\vol_n(\Lambda)$ is the $n$-dimensional volume of the face $\Lambda$

(ii) $A(\Lambda)$ is the exterior angle of the facet $\Lambda$, which is the angle of the dual cone $K_\Lambda$
of the facet $\Lambda$
(a full $n$-dimensional angle is considered equal to $1$).
Recall that $K_\Lambda$ consists of points $w\in\C^n$,
such that the linear functional $f(z)=\re z(w)$ as a function on the polyhedron $\Delta$ reaches a maximum simultaneously at all points of the face $\Lambda$.

(iii) $c(\Lambda)=\cos(T_\Lambda^\bot,\sqrt{-1}T_\Lambda)$ is the cosine of the angle between the subspaces $T_\Lambda^\bot$ and $\sqrt{-1}T_\Lambda$,
where $T_\Lambda\subset{\C^n}^*$ is the tangent space of the facet $\Lambda$, and $T_\Lambda^\bot\subset{\C^n}^*$ is the orthogonal complement of the space $T_\Lambda$.
Recall that the cosine of the angle between $n$-dimensional subspaces is, by definition, the distortion coefficient of the volume $\vol_n$ when projected from one to the other.
\end{definition}
\begin{example}\label{exPs1}
 (1) The pseudovolume of a polygon in $\C^1$ is equal to half of its perimeter, divided by $2\pi$.
 Here, the length of a line segment is considered to be twice its actual length.

(2) The pseudovolume of a polytope $\Delta\subset\re({\C^n}^*)$ is $\frac{1}{(2\pi)^n}\vol_n(\Delta)$.
\end{example}
\begin{definition}\label{dfPseudoMixed}
The mixed pseudovolume of convex polytopes $\Delta_1,\ldots,\Delta_n$ in the space ${\C^n}^*$ is defined as
$$
\mathfrak p(\Delta_1,\ldots,\Delta_n)=\frac{1}{(2\pi)^n}\sum_{\Lambda\subset\Delta,\,\dim(\Lambda)=n}
  c(\Lambda)\:A(\Lambda)\:\vol_n(\Lambda_1,\cdots,\Lambda_n),
$$
where the summation is over $n$-dimensional faces $\Lambda$ of the polyhedron-sum $\Delta=\Delta_1+\ldots+\Delta_n$
 (i.e. $\Delta$ is the Minkowski sum of the polyhedra $\Delta_i$),
 $\Lambda_i\subset\Delta_i$ -- faces-summands of the $n$-dimensional face $\Lambda$ 
(i.e. $\Lambda=\Lambda_1+\ldots+\Lambda_n$), and
$\vol_n(\Lambda_1,\cdots,\Lambda_n)$ is the $n$-dimensional mixed volume of the polytopes $\Lambda_i$.
\end{definition}
\begin{corollary}\label{corPseudoMixed1}
The function $\mathfrak p(\lambda_1\Delta_1 + \ldots + \lambda_n\Delta_n)$ of non-negative variables $\lambda_1,\ldots,\lambda_n$ is a homogeneous polynomial of degree $n$.
The mixed pseudovolume $\mathfrak p(\Delta_1,\ldots,\Delta_n)$ is equal to the coefficient of the monomial $\lambda_1 \cdot \ldots \cdot \lambda_n$, divided by $n!$.
\end{corollary}
\begin{corollary}\label{corPseudoMixed2}
The mixed pseudovolume is a symmetric multilinear $n$-form on the space of virtual convex polytopes in the space ${\C^n}^*$.
\end{corollary}
Let us remind that a virtual polytope is defined as the formal difference $\Delta - \Gamma$ of two defined up to a translation convex polytopes $\Delta,\Gamma$.
Virtual polytopes form a real vector space with operations of\ ``Minkowski addition'' and ``multiplication by number'',
where $\lambda \cdot (\Delta - \Gamma) = \text{sign}(\lambda) (|\lambda|\Delta - |\lambda|\Gamma)$.

For completeness of the presentation, we will provide the criteria for the vanishing of the mixed volume and mixed pseudovolume of polytopes.
\begin{definition}\label{dfPolRang}
For a polytope $\Delta\subset\R^n$, let $\dim(\Delta)$ denote the dimension of the minimal affine subspace,
containing $\Delta$.
The \emph{rank of a set of convex polytopes} $\Delta_1,\ldots,\Delta_m$ in the space $\R^n$ is defined as the minimum value $\dim(\Delta_{i_1}+\ldots+\Delta_{i_k}) -k$
over all $k\leqslant m$ and all subsets $1\leqslant i_1<\ldots<i_k\leqslant m$.
\end{definition}
\begin{proposition}\label{prPolRang}  {\rm\cite[Assertion 2.1.6]{EKK}}
The mixed volume of
$\Delta_1,\ldots,\Delta_n$ is zero if and only if the rank of the set of polytopes $\Delta_i$ is negative.
\end{proposition}
\begin{definition}\label{dfPolRangC}
For a polytope $\Delta\subset{\C^n}^*$, let $\dim_\C(\Delta)$ denote the dimension of the minimal complex affine subspace, containing $\Delta$.
The \emph{complex rank of a set of convex polytopes} $\Delta_1,\ldots,\Delta_m$ is defined as the minimum value $\dim_\C(\Delta_{i_1}+\ldots+\Delta_{i_k}) -k$
over all $k\leqslant m$ and all subsets $1\leqslant i_1<\ldots<i_k\leqslant m$.
\end{definition}
\begin{proposition}\label{prPolRangС}{\rm\cite[Corollary 3.3]{K14}}
The mixed pseudovolume of polytopes $\Delta_1,\ldots,\Delta_n$ in ${\C^n}^*$ is zero if and only if the complex rank of the set of polytopes $\Delta_i$ is negative.
\end{proposition}
\begin{corollary}\label{corPolRangC}
The product of the classes of numerical equivalence of $k$ \EA-hypersurfaces in the ring $\mathcal C(\C^n)$ is zero if and only if the complex rank of the set of their Newton polytopes is negative.
\end{corollary}
\begin{proof}
Let $\Delta_i$ be the Newton polytopes of the $k$ \EA-hypersurfaces $X_i$. According to \hyperlink{R3}{($\bf R_3$)} and \hyperlink{R5}{($\bf R_5$)},
we need to prove that the condition of negative complex rank for the set $\Delta_1, \ldots, \Delta_k$ is equivalent to the condition:
$$
\forall \Lambda_1, \ldots, \Lambda_{n-k} \subset {\C^n}^* : \mathfrak p(\Delta_1, \ldots, \Delta_k, \Lambda_1, \ldots, \Lambda_{n-k}) = 0.
$$
This follows from Definition \ref{dfPolRangC} and Proposition \ref{prPolRangС}.
\end{proof}
\begin{remark}\label{rmps2}
Let $A$ be a convex compact set in ${\C^n}^*$.
Let us remind that $h_A(z) = \max_{w \in A} \mathrm{Re}\, w(z)$ is called the support function of a compact convex set $A \subset {\C^n}^*$.
If $h_i$ is the support function of the set $A_i$,
then the current $dd^c h_1 \wedge \ldots \wedge dd^c h_n$ is well-defined and represents a non-negative measure in $\C^n$; see \cite{Dokl, K84, K14f}.
The integral of the measure $dd^c h_1 \wedge \ldots \wedge dd^c h_n$ over the unit ball in $\C^n$ with its center at the origin is called the mixed pseudovolume of the convex bodies $A_i$.
If $A_i$ are convex polytopes, then this integral, up to a constant factor depending on $n$, coincides with $\mathfrak p(A_1, \ldots, A_n)$.
In \cite{Al03}, it is proven that the pseudovolume is a translational and unitarily invariant valuation on convex bodies in ${\C^n}^*$.
The mixed pseudo-volume also arises in estimates of the asymptotic density of the set of common zeros of entire functions of exponential type \cite{KLast}.
\end{remark}
\subsection{Exponential Sums as Laurent Polynomials}\label{plan}
We consider exponen\-tial sums as Laurent polynomials on the infinite-dimensional character torus $\T_n^\infty$ of the additive group $\C^n_+$ in the space $\C^n$.
This approach can be traced back to the famous publication by Hermann Weyl \cite{Weyl}.

The torus $\T_n^\infty$ is the projective limit of the character tori $\T_G$ of finitely generated subgroups $G$ in the additive group of the space ${\C^n}^*$.
The projective limit of the standard windings of the tori $\T_G$
(see definition \ref{winding})
is the standard winding $\omega_\infty\colon\C^n\to\T_n^\infty$
of the infinite-dimensional torus $\T_n^\infty$.
The ring of exponential sums consists of the preimages of Laurent polynomials on $\T_n^\infty$ under the map $\omega_\infty$.

The rings of conditions $\mathcal C(\T_G)$ of the tori $\T_G$ form an inductive system of rings.
We denote the direct limit of the rings $\mathcal C(\T_G)$ by $\mathcal C(\T_n^\infty)$.
Using the representation of the ring $\mathcal C(\T_G)$ as the ring
generated by convex polyhedra in the character space of the torus,
we get a description of the ring $\mathcal C(\T_n^\infty)$ as a ring,
generated by finite-dimensional convex polyhedra in the infinite-dimensional character space of the torus $\T_n^\infty$.

Let $\mathfrak{R}$ denote the discrete additive group of the space ${\C^n}^*$.
Then the character space of the torus $\T_n^\infty$ is $\mathfrak{R} \otimes_\Z \R$.
Consider the extension of the identity map $\mathfrak{R} \to {\C^n}^*$
to a linear map of real vector spaces $\Pi \colon \mathfrak{R} \otimes_\Z \R \to {\C^n}^*$.
The action of the map $\Pi$ on convex polytopes in the character space extends to a homomorphism
of the ring of convex polytopes in $\mathfrak{R} \otimes_\Z \R$ to
the ring of convex polytopes $\mathfrak{S}_\vol({\C^n}^*)$
in the space ${\C^n}^*$ defined in \textsection\ref{polRing1}.
We construct
the ideal $J_\mathfrak{p}$ in the ring $\mathfrak{S}_\vol({\C^n}^*)$
(see \textsection\ref{polRing1}),
such that the ring of conditions of the space $\C^n$ is isomorphic to the quotient ring
$\mathfrak{S}_\vol({\C^n}^*) / J_\mathfrak{p}$;
see \textsection\ref{polRing1} and \textsection\ref{dw}, respectively.

The presentation of results is done at the finite-dimensional level.
We use the computation of the ring of conditions of the torus,
given in \cite{EKK}; see also \cite{KhA}.
One of the main concepts in the algebraic geometry of subvarieties of the torus is the Newton polytope of an algebraic hypersurface.
The computation of the ring of condition of the torus from \cite{EKK}
can be regarded as the construction of an analog of the Newton polytope for an arbitrary algebraic variety:
a variety is associated with an element of a certain ring,
generated by convex polytopes with vertices in the character lattice of the torus.
We call this element the \emph{Newtonization of an algebraic variety}.

In \textsection\ref{intro5}, any \EA-variety $X$ is associated with an algebraic variety in a certain multidimensional torus,
called the \emph{model of the \EA-variety} $X$.
In \textsection\ref{Newt}, we define the \emph{Newtonization of the \EA-variety} $X$ as the image
of the Newtonization of the model of $X$ under a certain ring homomorphism,
generated by convex polytopes;
see Theorem \ref{thmIdeals} and Corollary \ref{corM1}.
The Newtonization of an \EA-variety is an element of a certain ring generated by convex polytopes in ${\C^n}^*$,
defined in \textsection\ref{polRing}.

By analogy with \cite{EKK},
we interpret the Newtonization map as associating an \EA-variety with its numerical equivalence class,
i.e., the corresponding element in the ring of conditions of the space $\C^n$;
see \textsection\ref{EANewt}.
Using the Newtonizations of \EA-varieties,
we define in \textsection\ref{dw} the concepts of weak density and intersection index.

The definition of the intersection index of \EA-varieties given in \textsection\ref{dw} is not formally related in the usual way to the notions of the set of intersection points of \EA-varieties of complementary dimension.
The properties of the intersection index,
restoring the traditional context, are provided in \textsection\ref{index}.
Their proofs in \textsection\ref{pr1} heavily rely on the theory of toric varieties.
\section{Algebraic Codimension}\label{intro5}
In this section, we define some concepts that will be used later, including the concept of the algebraic dimension of \EA-varieties from \cite{K97}.
Further, we use the following notations throughout:

$\bullet$ $G$ - a finitely generated subgroup in the space $\Cn*$;
we assume that $G$ contains some basis of the space $\Cn*$

$\bullet$ $E_G$ - the ring of exponential sums with supports from the group $G$

$\bullet$ $G$-variety - \EA-variety defined by equations from $E_G$

$\bullet$ $\T_G$ - the character torus of the group $G$, $\C[\T_G]$ - the ring of Laurent polynomials on the torus $\T_G$

$\bullet$ $\mathfrak T_G$ - the Lie algebra of the torus $\T_G$, $\im \mathfrak T_G$ - the Lie algebra of the maximal compact subtorus in $\T_G$, $\re\mathfrak T_G=i\im \mathfrak T_G$

$\bullet$ $\mathfrak T_G^*$ - the space of linear functionals on $\mathfrak T_G$

$\bullet$ $\re \mathfrak T_G^*$ - the character space of the torus $\T_G$

$\bullet$ $\Z_G\subset\re\mathfrak T_G^*$ - the lattice of characters of the torus $\T_G$.
\begin{definition}\label{winding}
For $z\in\C^n$, we define the character $\omega_G(z)$ of the group $G$ as $\omega_G(z)\colon g\mapsto\E^{g(z)}$.
The map $\omega_G\colon\C^n\to\T_G$ is a group homomorphism and is called the \emph{standard winding map} to the torus $\T_G$.
\end{definition}
\begin{corollary}\label{corGvar1}
The pullback map $\omega_G^*\colon\C[\T_G]\to E_G$ is an isomorphism of rings.
\end{corollary}
\noindent
\emph{Proof.}
The image of the standard winding map is dense in the Zariski topology. This implies the desired statement.
\begin{definition} \label{dfModel}
Let $X$ be a $G$-variety defined by equations $\{P_i=0\}$, where $P_i\in E_G$.
The equations $\{(\omega_G^*)^{-1}P_i=0\}$ define
the algebraic variety $X_G\subset\T_G$ called the $G$-\emph{model} of the \EA-variety $X$.
\end{definition}
\begin{corollary}\label{corModel}
The map $X\mapsto X_G$ establishes a one-to-one correspondence between the set of $G$-varieties and the set of algebraic varieties in $\T_G$.
\end{corollary}
\begin{definition} \label{dfCodim}
The codimension of the $G$-model $X_G$ of the $G$-variety $X$ is denoted by $\codima X$ and is called the algebraic codimension of $X$.
Accordingly, $\dima X=n-\codima X$ is called the \emph{algebraic dimension} of $X$
(for example, the algebraic dimension of an \EA-variety in $\mathbb{C}^1$, defined by the equations $\E^{iz}-1 = \E^{2\pi i z}-1 = 0$, is $-1$).
\end{definition}
Note that for all groups $G$ such that the \EA-variety $X$ is a $G$-variety, the codimensions of the corresponding $G$-models $X_G$ are the same.
Therefore, the algebraic codimension of $X$ is well-defined.
The equidimensionality of one of the models of $X$ implies the equidimensionality of any of its $G$-models.
In this case, the \EA-variety $X$ is called equidimensional.
Furthermore, we default to the following assumptions:
1)
we consider the \EA-varieties to be equidimensional, and 2)
if a statement involves a finite set of \EA-varieties and a group $G\subset\Cn*$, then all of these \EA-varieties are $G$-varieties, i.e., they are defined by equations from the ring $E_G$.
For any finite set of \EA-varieties, such a group $G$ exists.
\begin{definition}\label{dfComplete}
An \EA-variety \( X \) with algebraic codimension $k$ is called an intersection of $k$ \EA-hyper\-surfaces if it's model \( X_G \) is an intersection of $k$ algebraic hypersurfaces.
\end{definition}
Note that the property of an \EA-variety being an
intersection does not depend on the choice of the group $G$.
\begin{proposition}\label{prSdvig1}
For equidimensional \EA-varieties $X$ and $Y$, there exists an \EA-hypersurface $Z$ such that for all \( z \notin Z \), the \EA-variety \( (z+X)\cap Y \) is equidimensional and \( \codim (z+X)\cap Y = \codim X + \codim Y \).
\end{proposition}
\begin{proof}
Let $X$ and $Y$ be $G$-varieries.
From algebraic geometry, it is known that for some algebraic hypersurface $\mathcal{Z} \subset \T_G$, for all $\tau \not\in \mathcal{Z}$,
the variety $(\tau X_G) \cap Y_G$ is equidimensional and $\codim (\tau X_G) \cap Y_G = \codim X_G + \codim Y_G$.
Then, for the \EA-hypersurface $Z$ such that $Z_G = \mathcal{Z}$, the required statements follow from the definition of
algebraic codimension.
\end{proof}
\noindent
\emph{Proof of Theorem} \ref{thRes1}.
Let $G$ be the group generated by the points of all supports $\Lambda_i$.
Each of the supports $\Lambda_i$ can be regarded as a finite set of torus characters of $\T_G$.
Accordingly, the space $S(\Lambda_1,\ldots,\Lambda_k)$ can be seen as the space $S$ consisting of Laurent polynomial systems on $\T_G$.
It is known from algebraic geometry that in the space $S$, there exists an algebraic hypersurface $D$ such that the following holds:
\emph{the codimension of the zero set of a system from the space $S$ not belonging to $D$ is an algebraic variety in $\T_G$ of codimen\-sion} $k$.
Let $F=0$ be the equation of the hypersurface $D$.
The polynomial $F$ corresponds to the exponential sum $\mathcal F$ in the space $S(\Lambda_1,\ldots,\Lambda_k)$.
Then, the \EA-hypersurface defined by the equation $\mathcal F=0$ has the required properties.
the theorem is proven.
\par\smallskip
Let us recall that a local irreducible component $(Y,z)$ of the \EA-variety $Y$ at the point $z$ is called atypical if
its codimension as an analytic set in $\C^n$ is less than $\codim Y$.
\begin{theorem}\label{thmAnom}
Let $(X,z)$ be an atypical component of the \EA-variety $X$ at the point $x\in X$,
and let the algebraic variety $Y$ in the torus $\mathbb{T}_G$ be a $G$-model of $X$.
Then there exist a proper subtorus $\mathbb{T}$ in $\mathbb{T}_G$ and $g\in\mathbb{T}_G$,
such that the image $\omega_G(X,z)$ of the component $(X,z)$ is contained in $g\mathbb{T}$.
\end{theorem}
\begin{corollary}\label{corZilber}
Any atypical component of an \EA-variety is contained in some proper affine subspace of $\C^n$.
\end{corollary}
\begin{proof}
The set $\omega_G^{-1}(\omega_G(z)\mathbb{T})$ is a countable union of parallel affine subspaces in $\mathbb{C}^n$.
From Theorem \ref{thmAnom}, it follows that one of them contains the component $(X,z)$.
\end{proof}
Theorem \ref{thmAnom} is proved in \cite{Z02}.
B. Zilber's proof relies on a well-known theorem proved by J. Ax \cite{Ax}.
This theorem is a functional analog of the famous Schanuel's conjecture on transcendence.
It is worth noting that the proof of some algebraic analogs of this statement is also based on the use of Ax's theorem; see \cite{BMZ07}.
%
%
\section{Rings Generated by Convex Polyhedra}\label{polRing}
\subsection{Definition of the Ring $\mathfrak S_\vol$}\label{polRing1}
Here is the definition of the ring generated by convex polytopes from \cite[Section 6]{EKK}.
Let's denote by ${\rm sym}(V)={\rm sym}_0(V)+{\rm sym}_1(V)+{\rm sym}_2(V)+\ldots$, where ${\rm sym}_0(V)=\R$, the graded symmetric algebra of the vector space $V$.
\begin{lemma}\label{lmPoly}
For any $k$, the vector space ${\rm sym}_k(V)$ is generated by elements of the form $\{v^k\colon v\in V\}$.
\end{lemma}
\begin{proof}
Follows from the commutativity of the ring ${\rm sym}(V)$.
\end{proof}
\begin{definition}\label{dfGor}
Suppose a symmetric multilinear $m$-form $\nu$ is defined on the space $V$.

(1) Let $I_\nu$ be a linear functional on the space ${\rm sym}_m(V)$ such that $$I_\nu(\Delta_1\cdot\ldots\cdot\Delta_m)=\nu(\Delta_1,\ldots,\Delta_m)$$
Extending the notation, we continue $I_\nu$ to a linear functional on the space ${\rm sym}(V)$,
setting it to zero on any homogeneous component ${\rm sym}_k(V)$ for $k\neq m$.

(2) Denote by $L_\nu(x,y)$ the symmetric bilinear form on the vector space ${\rm sym}(V)$ defined by $L_\nu(x,y)=I_\nu(x\cdot y)$.

(3) The kernel $J_\nu$ of the bilinear form $L_\nu$ is a homogeneous ideal of the graded ring ${\rm sym}(V)$.
Define the graded ring $\rm {sym}(\nu)(V)$ as $\rm {sym}(\nu)(V)={\rm sym}(V)/J_\nu$.
\end{definition}
The following proposition is a direct consequence of the definition.
\begin{proposition}\label{prpolRing}{\rm \cite[Section 6.1]{EKK}}
The following statements hold:

 {\rm(i)} $\rm {sym}_0(\nu)(V)=\R$

 {\rm(ii)} $\dim \rm {sym}_m(\nu)(V)=1$

 {\rm(iii)} The graded ring $\rm {sym}(\nu)(V)$ is generated by elements of degree $1$

 {\rm(iv)} The mappings $\rm {sym}_p(\nu)(V)\times\rm {sym}_{m-p}(\nu)(V)\to\R,$ defined as $(\eta,\xi)\mapsto I_\nu(\eta\cdot\xi)$, are non-degenerate pairings.
\end{proposition}
Let $V$ be the space of virtual convex polytopes in the vector space $E$.
Denote the graded symmetric algebra ${\rm sym}(V)$ by $S(E)=S_0+S_1+\ldots$.
Notice, that, by construction,
the space $S_1$ coincides with the space of virtual convex polytopes in $E$.
That is, the images of a polytope $\Delta$ and the translated polytope $e+\Delta$ in the space $S_1$ coincide,
and the multiplication of a polytope by a non-negative number is defined as the homothetic transformation.

For $\dim E = m$ on the space $V$ we consider a symmetric multilinear $m$-form $\vol$,
where $\vol(\Delta_1,\ldots,\Delta_m)$ is the mixed volume of the polyhedra $\Delta_1,\ldots,\Delta_m$.
In this case,
instead of ${\rm sym}(\vol)(V)$ from Definition \ref{dfGor} (3), we use the notation $\mathfrak S_\vol(E)$.

Furthermore, we also consider the case where $E={\C^n}^*$ and a symmetric multilinear $n$-form $\mathfrak p$,
where $\mathfrak p(\Delta_1,\ldots,\Delta_n)$ equals the mixed pseudo-volume of polytopes $\Delta_1,\ldots,\Delta_n$; see Definition \ref{dfPseudoMixed}.
In this case, the notation $\mathfrak S_\mathfrak p({\C^n}^*)$ is used instead of ${\rm sym}(\mathfrak p)(V)$.
When there is ambiguity in the notations, we write $S_k(E)$ and $J_\nu(E)$ instead of $S_k$ and $J_\nu$.

The following simple statement about the relationship between the ideals $J_\mu(E)$ and $J_\nu(E)$
for different multilinear forms $\mu$ and $\nu$ on the space of virtual polytopes in $E$ will be needed later.
\begin{lemma}\label{lmNu1}
Let $\mathfrak{n}$ and $\mathfrak{m}$ be $n$-linear and $m$-linear symmetric forms respect\-ively in the space of virtual polytopes in $E$.
Suppose $n\leq m$.
Then from $J_\mathfrak{m}(E)\cap S_n(E)\subset J_\mathfrak{n}(E)\cap S_n(E)$ it follows
that $J_\mathfrak{m}(E)\subset J_\mathfrak{n}(E)$.
\end{lemma}
\begin{proof}
If $k>n$, then by Definition \ref{dfGor}, $J_\mathfrak{n}(E)\cap S_k=S_k$,
i.e., $J_\mathfrak{m}(E)\cap S_k\subset J_\mathfrak{n}(E)\cap S_k$.
If $k<n$ and $\xi\in J_\mathfrak{m}(E)\cap S_k$,
then for all $\eta\in S_{n-k}$, $\xi\cdot\eta\in J_\mathfrak{m}(E)\cap S_n\subset J_\mathfrak{n}(E)\cap S_n$.
Hence, by Definition \ref{dfGor} (3), $\xi\in J_\mathfrak{n}(E)\cap S_k$. Thus, the lemma is proved.
\end{proof}
\begin{proposition}\label{prpolRingVol05}
The ring $\mathfrak{S}_\vol(E)$ does not depend on the choice of volume form in the space $E$.
\end{proposition}
\begin{proof}
Changing the volume form in the space $E$ multiplies the quadratic form $L_\vol$ from Definition \ref{dfGor} by a constant factor.
\end{proof}
Note that the structure of the ring $\mathfrak{S}_\mathfrak{p}({\C^n}^*)$ depends on the choice of Hermitian metric in the space ${\mathbb{C}^n}^*$.
\begin{proposition}\label{prpolRingVol}
Denote by ${\rm sym}_A\colon S(G)\to S(E)$
the homomorphism of symmetric algebras
corresponding to the linear operator $A\colon G\to E$.
For the rings $\mathfrak{S}_\vol$, the following holds.

{\rm(i)} For any linear operator $A\colon G\to E$, the homomorphism ${\rm sym}_A$
extends to a ring homomorphism $A_*\colon \mathfrak{S}_\vol(G)\to \mathfrak{S}_\vol(E)$,
such that $A_*(\Delta)=A(\Delta)$,
where $\Delta$ is any convex polytope in $E$.

{\rm(ii)}
If the operator $A$ is surjective or injective, then the ring homomorphism
$A_*$ is also respectively surjective
or injective.

{\rm(iii)}
For a linear operator $B\colon E\to U$, it holds that $B_*A_*=(BA)_*$.
\end{proposition}

\begin{proof}
If (i) is true,
then (ii) and (iii) are direct consequences of the definitions.
Since ${\rm sym}_B{\rm sym}_A={\rm sym}_{BA}$,
it is sufficient to prove (i) separately for the cases of injective and surjective operators $A$.
In each of these cases, statement (i) follows from the well-known property of mixed volumes given below in Proposition \ref{prFromEKK}.
\end{proof}
\begin{proposition}\label{prFromEKK}{\rm\cite[Proposition 2.1.7]{EKK}}
Let $A_1,\dots,A_k$ be convex bodies in the space $\R^N$.
Assume that these bodies
can be placed in a $k$-dimensional subspace $A\subset\R^N$ by parallel shifts.
Then for any convex bodies $A_{k+1},\ldots,A_N$, it holds that
$$
\vol_N(A_1,\ldots,A_N)=\vol_k(A_1,\ldots,A_k)\cdot\vol_{N-k}(\pi (A_{k+1}),\ldots,\pi (A_N)),
$$
where $\pi$ denotes the projection onto the orthogonal complement $A^\bot$ of the subspace $A$,
and $\vol_p$ denotes the mixed volume in a space of dimension $p$.
\end{proposition}
The main result in \textsection\ref{polRing} is the following statement.
\begin{theorem}\label{thmIdeals}
It is true that $J_\vol({\C^n}^*)\subset J_\mathfrak p({\C^n}^*)$.
\end{theorem}
\begin{corollary}\label{corM1}
The identity map of the ring $S({\C^n}^*)$
extends to a surjective homomorphism of rings
$\mathfrak S_\vol({\C^n}^*)\to\mathfrak S_\mathfrak p({\C^n}^*)$.
\end{corollary}
\noindent
According to Lemma \ref{lmNu1},
Theorem \ref{thmIdeals} follows from the following proposition.
\begin{proposition}\label{prIdeals}
It holds that $J_\vol({\C^n}^*)\cap S_n({\C^n}^*)\subset J_\mathfrak p({\C^n}^*)$.
\end{proposition}
The proof of Proposition \ref{prIdeals} is based on the description given below in Theorem \ref{thmIdealIntersection}
of the homogeneous components of the ideal $J_\vol(E)$, 
as well as on the construction 
of the weighted fan $\mathcal K_\mathfrak U$, corresponding to an element $\mathfrak U$ of the ring $\mathfrak S_\vol(E)$;
see Proposition \ref{prWeightedFan}.
This proof,
except for the application of tropical fans in \textsection\ref{tropTor}, \textsection\ref{pr1}, is not used further.
\subsection{Description of the ideal $J_\vol$}\label{polRing3}
Let $E^*$ be the space of linear functionals on an $m$-dimensional vector space $E$,
$K$ a convex polyhedral cone in $E^*$,
$V_K$ the subspace in $E^*$ generated by the points of the cone $K$,
and $V^\bot_K$ the orthogonal complement of the subspace $V_K\subset E^*$ in $E$.
Note that the highest nonzero component of the graded ring $\mathfrak S_\vol(V^\bot_K)$ has degree $m - \dim K$.
\begin{definition}\label{dfWeightedFan}
A fan $\mathcal K$ of cones of dimension $\leq k$ in the space $E^*$
is called a \emph{weighted fan of cones} if to each $k$-dimensional cone $K\in\mathcal K$,
an element $w(K)$ of the graded ring $\mathfrak S_\vol(V^\bot_K)$ of degree $m-k$ is assigned,
called the \emph{weight of the cone} $K$.
The union of all cones $K\in\mathcal K$ is denoted by $\supp\mathcal K$
and called the support of the fan $\mathcal K$.
Two $k$-dimensional weighted fans of cones $\mathcal{K,L}$ with the same support are considered equivalent
if $w(K)=w(L)$ for any two $k$-dimensional cones $K\in\mathcal K, L\in\mathcal L$ such that $\dim(K\cap L)=k$.
\end{definition}
%
%
Let $\Gamma$ be a convex polytope in the space $E$. We denote by $K_\Delta$ the set of linear functionals on the space $E$ that attain the maximum on $\Gamma$ simultaneously at all points of the face $\Delta\subset\Gamma$.
The convex polyhedral cone $K_\Delta$ is called the dual cone of the face $\Delta$, $\codim K_\Delta=\dim\Delta$.
The dual cones of faces of dimension $\geq m-k$ form a $k$-dimensional fan of cones $\mathcal K_{\Gamma,k}$ in the space $E^*$.
The fan of cones $\mathcal K_{\Gamma,m}$, consisting of dual cones of all faces, is called the dual fan of the polytope $\Gamma$.
For a $(m-k)$-dimensional face $\Delta$, we set $w(K_\Delta)$ equal to the image of the element $\Delta^{m-k}\in S_{m-k}(E^\bot_{V_K})$ in the ring $\mathfrak S_\vol(E/V^\bot_K)$.
We then consider $\mathcal K_{\Gamma,k}$ as a $k$-dimensional weighted fan with weights $w(K_\Delta)$.
\begin{remark}\label{rmOmegaNumber}
By definition, the weight $w(K)$ belongs to the component of degree $m-k$ of the graded ring $\mathfrak S_\vol(V^\bot_K)$. Since $\dim (V^\bot_K)=m-k$, the dimension of this component is one. Therefore, any chosen method of measuring volume in the space $V^\bot_K$ allows assigning a numerical value $\vol_{m-k}(w(K))$ to the weight $w(K)$.
\end{remark}
Now we define the addition of weighted $k$-dimensional fans $\mathcal{K,L}$.
The set $\supp(\mathcal K)\cup\supp(\mathcal L)$ is the union of subsets of three types:
$$
\{K\cap L\colon\,K\in{\mathcal K},L\in{\mathcal L}\},\,\,
\{K\setminus\supp(\mathcal L)\colon\,K\in\mathcal K\},\,\,
\{L\setminus\supp(\mathcal K)\colon\,L\in{\mathcal L}\}.
$$
The subsets of the first type are cones, and those of the second and third types are open subsets of cones.
We assign weights $w(K)+w(L)$ to $k$-dimensional subsets of the specified types,
where $w(K)$ and $w(L)$ are the weights of the cones $K$ and $L$.
Let ${\mathcal P}$ be a fan of cones with the support $\supp(\mathcal K)\cup\supp(\mathcal L)$,
such that each of the specified subsets is composed entirely of cones in the fan ${\mathcal P}$ (it is obvious that such fans exist).
We assign the weights to $k$-dimensional cones ${\mathcal P}$ containing them.
The sums of equivalent weighted fans are equivalent.
Thus, sums of equivalence classes of weighted fans are defined.
\begin{proposition}\label{prWeightedFan}
Let $\mathfrak U\in S_{m-k}(E)$ be represented as $\mathfrak U=\sum_{i\leq p}\pm \Gamma_i^{m-k}$, and let $\mathcal K_\mathfrak U=\pm\mathcal K_{\Gamma_1,k} +\ldots+\pm\mathcal K_{\Gamma_p,k}$. Then the $k$-dimensional weighted fan $\mathcal K_\mathfrak U$ does not depend on the choice of polytopes $\Gamma_i$. We call $\mathcal K_\mathfrak U$ the weighted fan of the element $\mathfrak U$.
\end{proposition}
To prove the proposition, we will utilize some additional statements related to the concept of a weighted fan.
Let $U^q(E)$ denote the set of ordered sets of $q$ linearly independent vectors $v_i$ from $E^*$.
Suppose that $\mathfrak{v}=(v_1,\ldots,v_{q})\in U^{q}(E)$,
$E_{\mathfrak{v}}$ is the subspace of common zeros of the functionals $v_i$ in $E$,
and $\iota_{\mathfrak{v}}\colon E_{\mathfrak{v}}\to E$ is the corresponding embedding operator.
We will continue to use the notation $\iota_{\mathfrak{v}}$ for the corresponding embedding of symmetric algebras $\iota_{\mathfrak{v}}\colon S(E_{\mathfrak{v}})\to S(E)$.
\begin{definition}\label{dfvolv}
Let $\mathfrak{v}=(v_1,\ldots,v_{m-k})\in U^{m-k}(E)$.
Denote by $\Delta_{\mathfrak{v}}$ the face of the convex polytope $\Delta$ consisting of points where the functionals $v_i$ attain their joint maximum on $\Delta$.
We retain the notation $\Delta_{\mathfrak{v}}$ for this face, shifted into the $k$-dimensional subspace $E_{\mathfrak{v}}$.
Thus, we can consider $\Delta_{\mathfrak{v}}$ as a virtual polytope in the space $E_{\mathfrak{v}}$.
Conversely, any virtual polytope in the space $E_{\mathfrak{v}}$
can be considered as a virtual polytope in $E$.
In particular, the polytope $\Delta_{\mathfrak{v}}$ is also a virtual polytope in $E$.
\end{definition}

\begin{lemma}\label{lmvol^k_2}
The mapping $\Delta \mapsto \Delta_{\mathfrak{v}}$ extends to homomorphisms $\mathfrak{U} \mapsto \mathfrak{U}_\mathfrak{v}$
of the rings $\pi_\mathfrak{v}\colon S(E) \to S(E_{\mathfrak{v}})$,
such that $\forall k\colon\pi_{\mathfrak{v}}(S_k(E)) \subset S_k(E_{\mathfrak{v}})$.
\end{lemma}
\begin{proof}
Since $(\lambda\Delta+\beta\Lambda)_\mathfrak v=\lambda\Delta_\mathfrak v+\beta\Lambda_ \mathfrak v$,
the statement follows from the definition of the symmetric algebra.
\end{proof}
Next, we retain the notation $\pi_{\mathfrak{v}}$ for the composition of homomorphisms of symmetric algebras:
$$
S(E) \xrightarrow{\pi_{\mathfrak{v}}} S(E_{\mathfrak{v}}) \xrightarrow{\iota_{\mathfrak{v}}} S(E).
$$
\begin{lemma}\label{lmvol^k_3}
Let $\mathfrak{v}=(v_1,\ldots,v_{m-k})\in U^{m-k}(E)$. For positive $\alpha_1,\ldots,\alpha_{m-k}$, consider $\mathfrak{u}=\alpha_1v_1+\ldots+\alpha_{m-k}v_{m-k}$ as an element in $U^1$. Then, for any $\mathfrak{U}\in S(E)$, we have $\pi_{\mathfrak{v}}(\mathfrak{U})=\pi_{\mathfrak{u}}(\mathfrak{U})$.
\end{lemma}
\noindent
\emph{Proof.}
This follows directly from Definitions \ref{dfvolv}.
\par\smallskip
\noindent
\emph{Proof of Proposition \ref{prWeightedFan}.} It follows from Lemmas \ref{lmvol^k_2} and \ref{lmvol^k_3}. We omit the details.
\begin{theorem}\label{thmIdealIntersection}
Let $1\leq k\leq m$ and $\mathfrak U\in S_k$.
Then the following conditions are equivalent:

{\rm(i)}\ $\mathfrak U\in J_\vol\cap S_k$

{\rm(ii)}\ All the weights of the cones in the weighted fan $\mathcal K_\mathfrak U$ are zero.
\end{theorem}
In the proof of the theorem, we assume that a scalar product $(*,*)$ is fixed in the space $E$.
In this context, the dual scalar product $(*,*)_*$ is considered in the space $E^*$.
Accordingly, for a $k$-dimensional polytope $\Delta$,
$\vol_k(\Delta)$ denotes its $k$-dimensional area.
Recall that, according to Definition \ref{dfGor} (3),
for $\mathfrak v\in U^{m-k}$ and $\mathfrak U\in S_k(E)$,
the image of the element $\pi_\mathfrak v(\mathfrak U)$
in the ring $\mathfrak S_\vol(E_\mathfrak v)$ is uniquely determined by the value of $\vol_k(\pi_\mathfrak v(\mathfrak U))$.
In particular, if $\vol_k(\pi_\mathfrak v(\mathfrak U))=0$,
then the image of $\pi_\mathfrak v(\mathfrak U)$ in the ring $\mathfrak S_\vol(E_\mathfrak v)$ is zero.

Below, we use a well-known property of the mixed volume of polytopes;
see, for example, \cite[Proposition 2.1.10]{EKK}.
\begin{lemma}\label{lmNSymm}

Let $\Upsilon,\Gamma_1,\ldots,\Gamma_{n-1}$ be polytopes in $E$, where $\dim E=m$,
and let $S_*$ be a unit sphere in $E^*$ centered in zero.
Then
$$
m!\:\vol(\Upsilon,\Gamma_1,\ldots,\Gamma_{n-1})=
\sum_{\mathfrak v\in  U^1(E)\cap S_*}h_\Upsilon(\mathfrak v)\: \vol_{m-1}(\pi_\mathfrak v (\Gamma_1),\ldots,\pi_\mathfrak v(\Gamma_{m-1})),
$$
where $h_\Upsilon$ is the support function of $\Upsilon$,
and $\vol_{m-1}(\pi_\mathfrak v (\Gamma_1),\ldots,\pi_\mathfrak v(\Gamma_{m-1}))$ denotes the mixed volume of polytopes in the space $E_\mathfrak v$.
\end{lemma}
\begin{corollary}\label{corNonSymm0}
  Let $\mathfrak U\in S_{m-1}$.
Then
$$
m!\:\vol(\Upsilon\cdot\mathfrak U)=
\sum_{\mathfrak v\in  U^1(E)\cap S_*}h_\Upsilon(\mathfrak v)\: \vol_{m-1}(\pi_\mathfrak v (\mathfrak U)),
$$
\end{corollary}
\begin{proof}
Since there exists a representation $\mathfrak U=\sum\pm\Gamma_i^{m-1}$,
the desired statement follows from Lemma \ref{lmNSymm}.
\end{proof}
\begin{corollary}\label{corNonSymm1}
  Let $\mathfrak U\in S_{m-1}(E)$.
  Suppose $\vol_{m-1}(\pi_\mathfrak v(\mathfrak U))=0$ for any nonzero $\mathfrak v\in U^1(E)$.
  Then $\mathfrak U\in J_\vol(E)$.
\end{corollary}
\begin{proof}
Follows from Corollary \ref{corNonSymm0}.
\end{proof}
\begin{corollary}\label{corNonSymm2}
  Let $\mathfrak U\in J_\vol\cap S_{m-1}(E)$.
 Then for any $\mathfrak v\in U^1(E)$, we have $\vol_{m-1}(\pi_\mathfrak v(\mathfrak U))=0$.
\end{corollary}
\begin{proof}
Recall that a fan $\mathcal K$ in $E^*$ is called complete if $\bigcup_{K\in\mathcal K}K=E^*$.
A positively homogeneous degree $1$ continuous function $h\colon E^*\to\R$,
linear on each of the cones of some complete fan in $E^*$, is called piecewise linear.
It is known that any piecewise linear function is the difference of the support functions of two convex polytopes in $E$.
If $\mathfrak U\in J_\vol\cap S_{n-1}(E)$,
then, by the definition of the ideal $J_\vol$,
for any polytope $\Gamma$, $\vol(\Gamma\cdot\mathfrak U)=0$.
Let $h=h_\Gamma-h_\Upsilon$.
Applying Corollary \ref{corNonSymm0} to the polytopes $\Gamma$ and $\Upsilon$,
we get the following:
if the function $h$ is piecewise linear, then
$\sum_{v\in S_*}h(v)\: \vol_v(\mathfrak U)=0$.
From this, the desired statement follows.
\end{proof}

\noindent\emph{Proof of Theorem \ref{thmIdealIntersection}.}
When $k=m$, the statement is a consequence of the definition of the ideal $J_\vol$.
When $k=m-1$, the equivalence of (i) and (ii) follows from Corollaries \ref{corNonSymm1} and \ref{corNonSymm2}.
Now let's prove the theorem for $k=p$, assuming it holds true for $k=p+1$.

Consider the following set of data:

$\bullet$\  $\mathfrak U\in S_{p}(E)$

$\bullet$\  $\mathfrak u\in U^{m-p-1}(E)$

$\bullet$\  $V=E_\mathfrak u$ -- $(p+1)$-dimensional subspace in $E$

$\bullet$\ $\mathfrak V=\pi_\mathfrak u(\mathfrak U)\in S_{p}(V)\subset S_{p}(E)$

$\bullet$\ $\Upsilon=\pi_\mathfrak u(\Gamma)$ -- convex polytope in $V$, where $\Gamma$ is a convex polytope in $E$.
\par\smallskip
First, suppose that $\mathfrak U\in J_\vol\cap S_p$, i.e., condition (i) holds.
Then $\Gamma\cdot\mathfrak U\in J_\vol\cap S_{p+1}$.
Therefore, applying (i)$\Rightarrow$(ii) for $k=p+1$,
we have that
$$
\forall \Gamma\colon\:\vol_{p+1}(\Upsilon\cdot\mathfrak V)=\vol_{p+1}(\pi_\mathfrak u(\Gamma\cdot\mathfrak U))=0.
$$
Hence, $\mathfrak V\in J_\vol(V)\cap S_{p}(V)$.
According to Corollary \ref{corNonSymm2},
$\vol_p(\pi_\mathfrak v(\mathfrak V))=0$ for any nonzero $\mathfrak v\in U^1(V)$.
Since, for any $\mathfrak z\in U^{m-p}$, there exist $\mathfrak u$ and $\mathfrak v$ such that $\pi_\mathfrak z =\pi_\mathfrak v\pi_\mathfrak u$, then
$\vol_p(\pi_\mathfrak z(\mathfrak V))=0$.
Thus, condition (ii) holds.

Now assume that condition (ii) holds.
Then for any $\mathfrak v\in U^1(V)$, we have $\vol_p(\pi_\mathfrak v(\mathfrak V))=0$.
According to Corollary \ref{corNonSymm0},
$$
\forall \Gamma\colon\: \vol_{p+1}(\pi_u(\Gamma\cdot\mathfrak U))=\Upsilon\cdot\mathfrak V\in J_\vol(V).
$$
Therefore,
applying (ii)$\Rightarrow$(i) for $k=p+1$,
we get $\Gamma\cdot\mathfrak U\in J_\vol$.
From this, the desired statement follows.
Theorem \ref{thmIdealIntersection} is proven.
\par\smallskip
By construction, the map $\mathfrak U \mapsto \mathcal K_\mathfrak U$ is a linear operator from the space $S_k(E)$ to the space of $(m-k)$-dimensional weighted fans.
From Theorem \ref{thmIdealIntersection}, the following statement follows.
\begin{corollary}\label{corWeightedOfU}
The kernel of the operator $\mathfrak U \mapsto \mathcal K_\mathfrak U$ is equal to $S_k(E)\cap J_\vol(E)$.
\end{corollary}
\noindent
\emph{Proof of Theorem \ref{thmIdeals}.}
Here we assume that a Hermitian metric is chosen in the space $\C^n$.
The notion of pseudo-volume depends on this choice.
Recall that the mixed pseudo-volume $\mathfrak p$ is considered as a symmetric multilinear $n$-form on the space of virtual polytopes in ${\C^n}^*$;
see 
Corollaries \ref{corPseudoMixed1}, \ref{corPseudoMixed2}.
Let $\mathfrak U\in J_\vol\cap S_n({\C^n}^*)$.
According to Proposition \ref{prIdeals},
Theorem \ref{thmIdeals} is reduced to proving the equality
$\mathfrak p(\mathfrak U)=0$; see definition \ref{dfGor}.
From definitions \ref{dfPseudo}, \ref{dfPseudoMixed}, it follows that $\mathfrak p(\mathfrak U)$ is equal to a linear combination of the form
\begin{equation}\label{eq_pFor}
 \sum_{K\in\mathcal K_\mathfrak U,\:\dim K=n}c_K\:\vol_n(w(K)),
\end{equation}
where $\mathcal K_\mathfrak U$ is the weighted fan of the element $\mathfrak U$ constructed above
(see Proposition \ref{prWeightedFan}),
and $w(K)$ is the weight of the cone $K$.
According to Theorem \ref{thmIdealIntersection},
all these weights are zero.
Theorem \ref{thmIdeals} is proven.
\subsection{Tropicalizations of elements of the ring $\mathfrak{S}_\vol$}\label{tropS}
Here, we describe the multiplication operation in the ring
$\mathfrak{S}_\vol(E)$ in terms of dual weighted fans in the space $E^*$.
In the case of $E = \re{\C^n}^*$,
this description practically coincides with the so-called tropical formula of Kushnirenko-Bernstein;
see \cite{EKK}.
It is not used further until \textsection\ref{tropTor}.

Below, we use the following notations:
$K$ -- a convex polyhedral cone in $E^*$,
$V_K$ -- the subspace in $E^*$
generated by the points of the cone $K$,
$V^\bot_K$ -- the annihilator of the subspace $V_K$ in $E$.
Let us recall some definitions
related to the concept of a weighted fan.

A fan of cones $\mathcal{K}$ of dimension $\leq k$ in the space $E^*$
is called a \emph{weighted fan of cones}
if each $k$-dimensional cone $K \in \mathcal{K}$
is associated with an element $w(K)$ of the graded ring $\mathfrak{S}_\vol(V^\bot_K)$ of degree $m-k$,
called the \emph{weight of the cone} $K$;
see definition \ref{dfWeightedFan}.
Note that
the highest nonzero component of the graded ring $\mathfrak{S}_\vol(V^\bot_K)$ has degree $m-k$
and is one-dimensional.
To any element $\mathfrak{U} \in S_k(E)$, we associate
the dual $(m-k)$-dimensional weighted fan $\mathcal{K}_\mathfrak{U}$ in the space $E^*$,
where $m=\dim(E)$;
see proposition \ref{prWeightedFan}.
Weighted fans,
dual to the elements of $S(E)$, will be referred to as \emph{tropical fans}.
\begin{definition}\label{dfTransvCone}
We call a pair of cones $(K, L)$ in the space $E^*$ transverse
if $V_K \cup V_L$ does not belong to some proper subspace in $E^*$.
If the pair $(K, L)$ is not transverse,
we set
\[
D(K, L) = \{e^* \in E^* \colon (e^* + V_K) \cap V_L \neq \emptyset\},
\]
otherwise we set $D(K, L) = \emptyset$.
The set $D(K, L)$ is a proper subspace in $E^*$.
Any point $e \not\in D(K, L)$ will be called \emph{admissible for the pair of cones} $(K, L)$.
If a point $e$ is admissible for every pair of cones $(K \in \mathcal{K}, L \in \mathcal{L})$,
then we call it \emph{admissible for the pair of fans} $(\mathcal{K}, \mathcal{L})$.
\end{definition}
Let $e$ be an admissible point
for a pair of fans $\mathcal{K}$ and $\mathcal{L}$ of dimensions respectively $p$ and $q$
in the $m$-dimensional space $E^*$.
We define the fan $\mathcal{K} \cap^e \mathcal{L}$ of dimension $(m - p - q)$,
called the $e$-intersection of the fans $\mathcal{K}$ and $\mathcal{L}$.
If the fans $\mathcal{K}$ and $\mathcal{L}$ are weighted,
then we define the weights of the $(m - p - q)$-dimensional cones in $\mathcal{K} \cap^e \mathcal{L}$.
If the fans $\mathcal{K}$ and $\mathcal{L}$ are tropical,
then the weighted fan is also tropical.

If $p + q < m$, then the fan $\mathcal{K} \cap^e \mathcal{L}$ is, by definition, empty.

If $p + q = m$, then
the fan $\mathcal{K} \cap^e \mathcal{L}$ is either empty or consists of a single cone $N = \{0\}$.
The weight of the zero cone $N$ is defined as follows.

\begin{definition}\label{dfmixedTropic}
Consider the set of pairs
$$
S(N) = \{(K, L) \colon K \in \mathcal{K}, \dim K = p, L \in \mathcal{L}, \dim L = q, (e + K) \cap L \ne \emptyset \}.
$$
In the equation (\ref{mixedTropic}) below,
the weights $w(K)$ and $w(L)$ are considered as elements of the ring $\mathfrak{S}_\vol(E)$.
We define the weight of the zero cone $N$ as
\begin{equation}\label{mixedTropic}
w(N) = \sum_{K \in \mathcal{K}, L \in \mathcal{L}, (e + K) \cap L \ne \emptyset} w(K) \cdot w(L)
\end{equation}
\end{definition}

\begin{example}\label{exProd1}
Let $\mathcal{K}$ and $\mathcal{L}$ be weighted fans of convex polygons $\Gamma$ and $\Upsilon$ in the plane $\R^2$.
In this case, $w(N)$ is an element of the one-dimensional space
consisting of elements of the second degree in the ring $\mathfrak{S}_\vol(\R^2)$.
By choosing a scalar product in $\R^2$, the weight $w(N)$ takes on a numerical value.
The following statement is a problem of school-level geometry:
the weight $w(N)$ is equal to the mixed area of the polygons $\Gamma$ and $\Upsilon$.
In particular,
if $\Gamma = \Upsilon$,
then the weight $w(N)$ is equal to the area of the polygon $\Gamma$.
\end{example}
Consider the case $p + q \geq m$.
Denote by $\mathcal{K} \cap \mathcal{L}$ the fan
consisting of pairwise intersections of cones from the fans $\mathcal{K}$ and $\mathcal{L}$.
For example, $\mathcal{K} \cap \mathcal{K} = \mathcal{K}$.
We define the $(p + q - m)$-dimensional subfan $\mathcal{K} \cap^e \mathcal{L}$ of the fan $\mathcal{K} \cap \mathcal{L}$,
depending on the choice of a point $e$ that is admissible for the pair $(\mathcal{K}, \mathcal{L})$.
For this,
we will need the concept of fan factorization.
\begin{definition}\label{dfFact}
Let $K \in \mathcal{K}$,
and $U$ be a vector subspace in $V_K$.
The image $L_U$ of a cone $L \in \mathcal{K}$
under the projection mapping $E^* \to E^*/U$
is called the $U$-factorization of the cone $L$.
All $U$-factorizations $L_U$ of cones $L\in \mathcal{K}$,
such that $L\supseteq K$, 
form a fan of cones $\mathcal{K}_U^{{\rm fact}}$
in the space $E^*/U$,
called the $U$-factorization of the fan $\mathcal{K}$.
If a $k$-dimensional fan $\mathcal{K}$ is weighted,
and $\dim L = k$,
we assign the weight $w(L_U)$ equal to the weight of the cone $L$ in $\mathcal{K}$.
\end{definition}
\begin{corollary}
Let $K \in \mathcal{K}$,
$L \in \mathcal{L}$,
$U \subseteq V_K \cap V_L$,
and the point $e$ be admissible for the pair $(\mathcal{K}, \mathcal{L})$.
Then the image $e_U$ of the point $e$ in $E^*/U$ is admissible
for the pair of fans $(\mathcal{K}_U^{{\rm fact}}, \mathcal{L}_U^{{\rm fact}})$.
\end{corollary}
\begin{definition}\label{dfeIntersect}
Let $e$ be an admissible point for the pair of cones $(\mathcal{K},\mathcal{L})$,
$K\in \mathcal{K}$, $L\in \mathcal{L}$, $\dim K=\dim\mathcal{K}=p$, $\dim L=\dim\mathcal{L}=q$,
$M=K\cap L \in \mathcal{K}\cap\mathcal{L}$,
$U=V_K\cap V_L$.
Then

(1) $M\in\mathcal{K}\cap^e\mathcal{L}$ if and only if $\mathcal{K}_U^{{\rm fact}}\cap^{e_U}\mathcal{L}_U^{{\rm fact}}\ne\emptyset$.
We define the fan $\mathcal{K}\cap^{e}\mathcal{L}$ as the fan
consisting of all such cones and their faces.

(2) If the fans $\mathcal{K}$, $\mathcal{L}$ are weighted,
then, using the rule (\ref{mixedTropic}),
we define the weight of the cone $M$ as the weight of the zero cone in $\mathcal{K}_U^{{\rm fact}}\cap^{e_U}\mathcal{L}_U^{{\rm fact}}$.
\end{definition}
\begin{definition}\label{dfEquiWeightedTrop}
The equivalence class of weighted fans of an element $\mathfrak{U}\in\mathfrak{S}_\vol(E)$
is called the \emph{tropicalization} of the element $\mathfrak{U}$.
Any weighted fan from this class is called a \emph{tropical fan} of the element $\mathfrak{U}$.
\end{definition}
\begin{theorem}\label{thmBKK_Trop1}
Let the weighted fans $\mathcal{K}$, $\mathcal{L}$ be tropical.
Then the following statements are true.

{\rm(1)} The weighted fan $\mathcal{K}\cap^e\mathcal{L}$ does not depend on the choice of the admissible point $e$ for the pair of fans,
and it is also tropical.
Thus, the operation $\mathcal{K}\cap^e\mathcal{L}$ can be considered as the product $\mathcal{K}\cdot\mathcal{L}$ of tropical fans $\mathcal{K}$ and $\mathcal{L}$.

{\rm(2)}
If $\mathcal{K}$, $\mathcal{L}$ are tropical fans of elements $\mathfrak{U}$, $\mathfrak{V}$ of the ring $\mathfrak{S}_\vol(E)$,
then $\mathcal{K}\cdot\mathcal{L}$ is a tropical fan of the element $\mathfrak{U}\cdot\mathfrak{V}\in\mathfrak{S}_\vol(E)$.
Thus, the operation of multiplication of tropicalizations of elements of the ring $\mathfrak{S}_\vol(E)$ is well defined.

{\rm(3)}
The assignment of the tropicalization of an element $\mathfrak{U}\in\mathfrak{S}_\vol(E)$ is an isomorphism of rings.
\end{theorem}
We do not provide the proof of the theorem.
If, for $E=\R^n$, we consider, instead of the ring $\mathfrak{S}_\vol(\R^n)$, its subring generated by polytopes with vertices at lattice points,
then the proof is given in \cite{EKK}.
If we choose an arbitrary integer lattice in $E$,
then, using continuity arguments,
we can obtain the desired statements for any polytopes in $E$.
An independent proof can be found in \cite{K03}.

Recall that to a linear operator $\pi\colon U\to V$ corresponds a ring homomorphism $\pi_*\colon\mathfrak S_\vol(U)\to\mathfrak S_\vol(V)$;
see Proposition \ref{prpolRingVol}.
According to Theorem \ref{thmBKK_Trop1} (3),
this homomorphism can be described in terms of tropical fans.
Here, we provide this description in a particular case,
for further use in \textsection\ref{tropEA}.

Let $\iota$ be the embedding of the space $L$ into $E^*$.
For simplicity of formulation, as well as for convenience of application in \textsection\ref{tropEA},
let us assume that a Hermitian metric is fixed in the space $L$.
Let $\pi\colon E\to L^*$ be the adjoint operator to $\iota$.
Let $\mathfrak U$ be a homogeneous element in the ring $\mathfrak S_\vol(E)$.
Thus, $\pi_*(\mathfrak U)\in\mathfrak S_\vol(L^*)$.
Let $\mathcal E=\supp(\pi_*(\mathfrak U))$ denote the support of the tropical fan of the element $\pi_*(\mathfrak U)$,
i.e., the \emph{fan of cones without weights} in the space $L$.
\begin{proposition}\label{prConnection}
Let $\mathcal U$ be the support of the weighted fan of the element $\mathfrak U$,
i.e., $\mathcal U$ is a fan of cones in the space $E^*$.
Let $\mathcal L$ be the fan of cones in the space $E^*$ consisting of a single cone $L$.
Then, for any admissible point for the pair of fans $(\mathcal U,\mathcal L)$ in $E^*$, we have $\mathcal E=\mathcal U\cap^e\mathcal L$.
\end{proposition}
The proof of the proposition is omitted.
It consists of several references to definitions and statements from \textsection\ref{polRing3}, \textsection\ref{tropS}.

\section{Newtonizations of Varieties}\label{Newt}
\subsection{Newtonizations of Algebraic Varieties}\label{AlgNewt}
We provide a description of the ring of conditions $\mathcal{C}(\T)$ of the $n$-dimensional complex torus $\T$ from \cite{EKK}.
Let $\mathfrak{T}$ be the Lie algebra of the torus $\T$,
$\re\mathfrak{T}=i\im\mathfrak{T}$,
where $\im\mathfrak{T}\subset\mathfrak{T}$ is the Lie algebra of the maximal compact subtorus in $\T$.
Let $\Z^n$ denote the character lattice of the torus $\T$ in the space of characters $\re\mathfrak{T}^*$
(i.e., in the space of linear functionals on $\re\mathfrak{T}$).
The Newton polytopes of Laurent polynomials are convex polytopes in $\re\mathfrak{T}^*$,
with vertices belonging to the lattice $\Z^n$.
Denote by $\mathfrak{S}_\vol(\Z^n)$ the subring of the ring $\mathfrak{S}_\vol(\re\mathfrak{T}^*)$ generated by Newton polytopes.
\begin{theorem}\label{ringEKK} {\rm\cite[Theorem 6.3.1]{EKK}}
For any equidimensional algebraic variety $X\subset\T$ of codimension $k$, there exists an element $\nwt(X)\in\left(\mathfrak{S}_\vol(\Z^n)\right)_k$,
hereinafter called the Newtonization of the variety $X$,
such that

{\rm(1)} for a hypersurface $X$ defined by the equation $f=0$,
$\nwt(X)=\Delta(f)$, where $\Delta(f)$ is the Newton polytope of the polynomial $f$

{\rm(2)} for numerically equivalent
varieties $X$ and $Y$, $\nwt(X)=\nwt(Y)$

{\rm(3)} if we retain the notation $\nwt(\bar{X})$ for the numerical equivalence class $\bar{X}$ of the variety $X$,
then the map $\nwt\colon\mathcal{C}(\T)\to \mathfrak{S}_\vol(\Z^n)$
is an isomorphism of graded rings.
\end{theorem}
Recall that the homogeneous component $\left(\mathfrak{S}_\vol(\re\mathfrak{T}^*)\right)_n$
of degree $n$ of the ring $\mathfrak{S}_\vol(\re\mathfrak{T}^*)$ is one-dimensional.
The choice of a linear functional $\kappa$ on the space $\left(\mathfrak{S}_\vol(\re\mathfrak{T}^*)\right)_n$
allows us to identify
an element $\chi\in\left(\mathfrak{S}_\vol(\re\mathfrak{T}^*)\right)_n$ with a real number $\kappa(\chi)$.
\begin{corollary}\label{corBKK}
Let $\C[x_1,x_1^{-1},\ldots,x_n,x_n^{-1}]$ be the ring of polynomials on the torus $\T$.
Denote by $\chi_i$ the Newton polytope of the polynomial $x_i-1$.
Choose a linear functional $\kappa$ on the one-dimensional space $\left(\mathfrak{S}_\vol(\re\mathfrak{T}^*)\right)_n$ such that
$\kappa(\chi_1\cdot\ldots\cdot\chi_n)=1$.
Then

{\rm(1)} if $\dim X + \dim Y = n$,
then the intersection index of the varieties $X$ and $Y$ is equal to $\kappa\left(\nwt(X)\cdot\nwt(Y)\right)$

{\rm(2)}
$\kappa(\Gamma_1\cdot\ldots\cdot\Gamma_n)=n!\vol(\Gamma_1,\ldots,\Gamma_n)$.
\end{corollary}
\begin{proof}
According to Proposition \ref{prpolRing} (iv),
the ring of conditions is generated by hypersurfaces.
Therefore, both statements follow from the Kushnirenko-Bernstein-Khovanskii theorem.
\end{proof}
In \textsection\ref{EANewt}, the concept of the Newtonization of \EA-varieties will be defined;
see Definitions \ref{dfNewtonizationET} and \ref{dfNewtonizationE}.
For this purpose, the following properties of the Newtonization of algebraic varieties are used.
\begin{corollary}\label{corNewt1}
The following holds:

{\rm(1)} $\forall \tau\in\T\colon\: \nwt(\tau X)=\nwt(X)$

{\rm(2)} For any varieties $X,Y\subset\T$,
there exists an algebraic hypersurface $D(X,Y)\subset\T$
such that the following is true: if $\tau\not\in D(X,Y)$, then $\nwt(\tau X\cap Y)=\nwt(X)\cdot\nwt(Y)$.
\end{corollary}
\noindent
\emph{Proof.}
Theorem \ref{ringEKK} reduces the statement to the standard properties of the ring of conditions of the torus $\T$.

\begin{proposition}\label{prNewtA}
Let $F\colon\T\to\A$ be a homomorphism of complex tori,
$dF\colon\mathfrak T\to\mathfrak A$ be the differential of the map $F$ at the identity point,
and $\mathfrak f\colon \re\mathfrak A^*\to\re\mathfrak T^*$ be the restriction of the linear operator adjoint to the differential $dF$ on the character space.
For an algebraic variety $A$ in the torus $\A$,
denote by $B$ its preimage under the homomorphism $F$.
Then
$\mathfrak f_*(\nwt(A))=\nwt(B)$,
where the ring homomorphism $\mathfrak f_*\colon\mathfrak S_\vol(\re\mathfrak A^*)\to\mathfrak S_\vol(\re\mathfrak T^*)$ is defined in {\rm\textsection\ref{polRing}};
see Proposition {\rm\ref{prpolRing}}.
\end{proposition}
\noindent
\emph{Proof.}
For $\codim A=1$, it follows from the standard properties of Newton polytopes.
According to Proposition \ref{prpolRing} (iv),
the required statement is reduced to the case $\codim A=1$.

\subsection{Newtonizations of \EA-Varieties}\label{EANewt}
The notations introduced in \textsection\ref{polRing} and at the beginning of \textsection\ref{intro5} are used here.
Let $\omega^G\colon\mathfrak T^*_G\to{\C^n}^*$ denote the linear operator
adjoint to the differential $d\omega_G\colon\C^n\to\mathfrak T_G$ of the homomorphism of the standard winding $\omega_G\colon\C^n\to\T_G$ at the point $0\in\C^n$.
Further, we retain the notation $\omega^G$ to denote the restriction of the map $\omega^G$
to the character space $\re\mathfrak T^*_G$ of the torus $\T_G$.
\begin{lemma}\label{lmZG}
{\rm(1)} It is true that the restriction of the map $\omega^G$ to the character lattice $\Z_G$ of the torus $\T_G$ is an isomorphism of groups $\omega^G\colon\Z_G\to G$.

{\rm(2)} The map $\Gamma\mapsto \omega^G(\Gamma)$ is a surjective map from the space of virtual polytopes in the character space $\mathfrak T^*_G$ with vertices in the lattice $\Z_G$ to the space of virtual polytopes in ${\C^n}^*$ with vertices in the group $G$.
\end{lemma}
\noindent
\emph{Proof.}
According to the Pontryagin duality theorem, the groups $G$ and $\Z_G$ are isomorphic.
It is easy to verify that this isomorphism can be established using the map $\omega^G$.
This implies statement (1).
Statement (2) follows from the linearity of the map $\omega^G\colon\re\mathfrak T^*_G\to{\C^n}^*$.
\begin{corollary}\label{corNewtE0}
Let $X$ be an \EA-hypersurface with the equation $f=0$, where $f\in E_G$, $\gamma$ is the Newton polytope of the exponential sum $f$, and $X_G$ is the $G$-model of the \EA-variety $X$.
Denote by $\Gamma$ the Newton polytope of the algebraic hypersurface $X_G$.
Then $\omega^G(\Gamma)=\gamma$.
\end{corollary}
\begin{proof}
It follows from Lemma \ref{lmZG}.
\end{proof}
According to Proposition \ref{prpolRingVol},
we have the ring homomorphism
\begin{equation}\label{eqInto}
  (\omega^G)_*\colon\mathfrak S_\vol(\re\mathfrak T^*_G)\to\mathfrak S_\vol({\C^n}^*).
\end{equation}
\begin{corollary}\label{corNewtE0.5}
 Let $\mathfrak S_\vol(\Z_G)$ be the subring of $\mathfrak S_\vol(\re\mathfrak T_G^*)$ generated by polytopes in $\re\mathfrak T_G^*$ with vertices in $\Z_G$, and $\mathfrak S_\vol(G)$ be the subring of $\mathfrak S_\vol({\C^n}^*)$ generated by polytopes in ${\C^n}^*$ with vertices in $G$. Then $(\omega^G)_*(\mathfrak S_\vol(\Z_G))=\mathfrak S_\vol(G)$.
\end{corollary}
\begin{proof}
 It follows from Corollary \ref{corNewtE0}.
\end{proof}
Let $G, H$ be subgroups with a finite number of generators in ${\C^n}^*$, and suppose $G\subset H$.
Denote by $\rho_{H,G}\colon\T_H\to\T_G$ the homomorphism of character restrictions from $H$ to $G$.
Let $\rho^{H,G}\colon\mathfrak T^*_G\to\mathfrak T^*_H$ be the linear operator conjugate to the differential $d\rho_{H,G}\colon\mathfrak T_H\to\mathfrak T_G$ of the homomorphism $\rho_{H,G}$.
According to Proposition \ref{prpolRingVol}, the operator $\rho^{H,G}$ corresponds to a ring homomorphism $\rho^{H,G}_*\colon \mathfrak S_\vol(\mathfrak T^*_G)\to\mathfrak S_\vol(\mathfrak T^*_H)$.
The map $\rho^{H,G}_*$ maps the character lattice $\Z_G$ of the torus $\T_G$ to the character lattice $\Z_H$ of the torus $\T_H$.
(This can also be described as taking the preimage of a character under the torus homomorphism $\rho_{H,G}\colon\T_H\to\T_G$.)
It follows that the ring homomorphism $\rho^{H,G}_*$, when restricted to the subring $\mathfrak S_\vol(\Z_G)$, induces a ring homomorphism $\rho^{H,G}_*\colon\mathfrak S_\vol(\Z_G)\to\mathfrak S_\vol(\Z_H)$.

 \begin{lemma}\label{lmNewtE1}
 The following diagram is commutative
 \[\begin{tikzcd}
\mathfrak S_\vol(\Z_G)\arrow{r}{\rho^{H,G}_*}\arrow{d}{(\omega^G)_*}&\mathfrak S_\vol(\Z_H)\arrow{d}{(\omega^H)_*}
\\ \mathfrak S_\vol(G)\arrow[hook]{r}{}& \mathfrak S_\vol(H)\arrow[hook]{r}{}&\mathfrak S_\vol({\C^n}^*)
\end{tikzcd}\]
\end{lemma}
\begin{proof}
It follows from the commutativity of the diagram
  \[\begin{tikzcd}
\C^n\arrow{r}{\omega_H}\arrow{dr}{\omega_G}&\T_H\arrow{d}{\rho_{H,G}}
\\ &\T_G,
\end{tikzcd}\]
which is a consequence of the definition of the standard winding.
\end{proof}
Let $X$ be an algebraic variety in the torus $\T_G$,
and $Y=\rho_{G,H}^{-1}(X)$.
Recall that the Newtonizations $\nwt (X)$ and $\nwt (Y)$ of the algebraic varieties $X$ and $Y$
are elements of the subrings $\mathfrak S_\vol(\Z_G)$ and $\mathfrak S_\vol(\Z_H)$
respectively, in the rings $\mathfrak S_\vol(\re\mathfrak T^*_G)$ and $\mathfrak S_\vol(\re\mathfrak H^*_G)$;
see Theorem \ref{ringEKK}.
 \begin{corollary}\label{corNewtE1}
  It is true that $\omega^G_*(\nwt (X))=\omega^H_*(\nwt (Y))$.
\end{corollary}
  \noindent
 \emph{Proof.}
From Proposition \ref{prNewtA} it follows that
 $\rho^{G,H}_*(\nwt (X))=\nwt (Y)$.
The required statement follows from Lemma \ref{lmNewtE1}.
 \begin{corollary}\label{corNewtE2}
 Let $G\subset H$ and the \EA-variety $X$ be defined by equations from the ring $E_G$.
 Then for the $G$-model $X_G$ and the $H$-model $X_H$ of the \EA-manifold $X$ it is true,
 that
 $\omega^G_*(\nwt (X_G))=\omega^H_*(\nwt (X_H))$.
\end{corollary}
\begin{proof}
Since $X_H=\rho_{G,H}^{-1}(X_G)$,
then the statement follows from Corollary \ref{corNewtE1}.
\end{proof}
\begin{definition}\label{dfNewtonizationET}
Let the \EA-variety $X$ be a $G$-variety.
Then we set $\nwt(X)=\omega^G_*(\nwt(X_G))$ and call the element $\nwt(X)$ of the ring $\mathfrak S_\vol({\C^n}^*)$
\emph{$A$-Newtoni\-zation of the \EA-variety} $X$.
According to corollary \ref{corNewtE2},
the $A$-Newtonization of an \EA-variety does not depend on the choice of group $G$.
\end{definition}
\begin{definition}\label{dfNewtonizationETtrop}
Tropicalization of the element $\nwt(X)\in\mathfrak S_\vol({\C^n}^*)$
is called \emph{the tropicalization of the \EA-variety} $X$;
see definition \ref{dfEquiWeightedTrop}.
Any tropical fan
defining the structure of the tropical variety  $\nwt(X)$
(see definition \ref{dfEquiWeightedTrop})
is called \emph{the tropical fan of the \EA-variety} $X$.
\end{definition}
\begin{corollary}\label{corEtropDim}
If $X$ is an \EA-variety in $\C^n$ of algebraic dimension $k$,
then the dimension of its tropical fan is $n+k$.
\end{corollary}
\begin{proof}
According to the definition of \ref{dfNewtonizationET},
$\nwt(X)=\omega_*^G(\nwt(X_G))$.
Because $\codim(X_G)=n-k$,
then the degree of the element $\nwt(X)$ in the ring $\mathfrak S_\vol({\C^n})$ is equal to $n-k$.
By definition of a weighted fan $K_{\nwt(X)}$
(see Proposition \ref{prWeightedFan}) we get,
that $\dim K_{\nwt(X)}=2n-(n-k)=n+k$.
\end{proof}
\begin{corollary}\label{corEtropQuasi}
Let $X$ be a quasi-algebraic \EA-variety of algebraic dimension $0$.
Then its tropical fan can be chosen to consist of a single $n$-dimensional cone $\im(\C^n)$.
\end{corollary}
\noindent
\emph{Proof.}
From the definition of quasi-algebraicity it follows that
$\nwt(X)$ is an element of degree $n$ in the graded subring $\mathfrak S_\vol(\re{\C^n}^*)$
in $\mathfrak S_\vol({\C^n}^*)$.
From here the necessary statement follows.
\begin{definition}\label{dfNewtonizationE}
We denote by $X^\nwt$ the image of the $A$-Newtonization $\nwt(X)$ under the ring homomorphism
$\mathfrak S_\vol({\C^n}^*)\to\mathfrak S_\mathfrak p({\C^n}^*)$ from Corollary \ref{corM1}.
The element $X^\nwt\in\mathfrak S_\mathfrak p({\C^n}^*)$ is called \emph{the Newtonization of \EA-variety} $X$.
\end{definition}
\begin{example}\label{exAnewt1}
Let $X=\{z\in\C^1\colon f(z)=0\}$,
where $f$ is an exponential sum in $\C^1$.
Then

(1)
the $A$-Newtonization $\nwt(X)$ \EA-variety $X$ is Newton polygon $\Delta$ of exponential sum $f$

(2)
the tropical fan of $X$ is formed by the rays of the external normals to the sides of the polygon $\Delta$;
the weights of the cones are equal to the lengths of the corresponding sides

(3)
the $A$-Newtonization $X^{\nwt}$ is equal to the half-perimeter of $\Delta$.
\end{example}
\begin{example}\label{exAnewt2}
Let the \EA-variety $X$ be quasi-algebraic.
Then, by construction, $\nwt(X)$ belongs to the subring $\mathfrak S_\vol(\re {\C^n}^*)$
of the ring $\mathfrak S_\vol({\C^n}^*)$.
We denote the corresponding element of the ring by $\mathfrak S_\vol(\re {\C^n}^*)$
 via $\nwt_\R(X)$.
The tropical fan $\mathcal K_\R$ of the element $\nwt_\R(X)$ is a weighted fan of cones in the space $\re(\C^n)$.
Then the tropical fan of the element $\nwt(X)$ consists of cones of the form $K+\im\C^n$,
where $K\in\nwt_\R(X)$.
\end{example}
\begin{corollary}\label{corAnyPol}
{\rm(1)} Any element of the ring $\mathfrak S_\vol({\C^n}^*)$ is an $A$-Newtonization
of some \EA-variety.

{\rm(2)} Any element of the ring $\mathfrak S_\mathfrak p({\C^n}^*)$ is a Newtonization
of some \EA-variety.

{\rm(3)}
Any element of the subring $\mathfrak S_\vol(\re {\C^n}^*)$ in the ring $\mathfrak S_\vol({\C^n}^*)$
is an $A$-Newtonization of some quasi-algebraic \EA-variety.
\end{corollary}
\noindent
\emph{Proof.}
Vertices of any convex polyhedron in ${\C^n}^*$
generate some subgroup $G\subset{\C^n}^*$.
Therefore, statements (1) and (2) follow from Lemma \ref{lmZG} and from Proposition \ref{prpolRing} (iv).
In the quasi-algebraic case, a subgroup $G$ can be chosen in the space $\re{\C^n}^*$.
Statement (3) follows from this.
\begin{theorem}\label{thmNewtShift}
{\rm(1)} For any $z\in\C^n$ it is true that $\nwt(z+X)=\nwt(X)$ and $(z+X)^\nwt=X^\nwt$

{\rm(2)} for \EA-varieties $X,Y$
there is an \EA-hypersurface $\mathfrak D(X,Y)$,
such that if $z\not\in\mathfrak D(X,Y)$ then $\nwt((z+ X)\cap Y)=\nwt(X)\cdot\nwt(Y)$
and $((z+ X)\cap Y)^\nwt=X^\nwt\cdot Y^\nwt$.
\end{theorem}
\begin{proof}
According to Proposition \ref{prpolRing},
the rings $\mathfrak S_\vol({\C^n}^*)$ and $\mathfrak S_\mathfrak p({\C^n}^*)$ are generated by virtual polyhedra.
Therefore, the necessary state\-ments follow from the corollary \ref{corNewt1} for $A=X_G$, $B=Y_G$ and from
$\mathfrak D(X,Y)=\omega_G^{-1} D(X_G,Y_G)$.
\end{proof}
\section{Intersection Index and the Ring of Conditions of $\C^n$}\label{dw}
Here we define the intersection index and consider the corresponding concept of numerical equivalence of \EA-varieties.
It is proved that the classes of numerical equivalence form
graded ring,
called the ring of conditions of the space ${\C^n}$.
This  ring
is isomorphic to the ring $\mathfrak S_\mathfrak p({\C^n}^*)=\mathfrak S_0+\ldots+\mathfrak S_n$ defined in \textsection\ref{polRing}.
The mentioned isomorphism is established using the Newtonization map of \EA-varieties defined in \textsection\ref{EANewt};
see Definitions \ref{dfNewtonizationET}, \ref{dfNewtonizationE}.

Recall that the \EA-variety $X$ is associated with homogeneous elements
$\nwt (X)$ and $X^\nwt$ degrees $\codima X$ in
graded
rings
$\mathfrak S_\vol{(\C^n}^*)$ and $\mathfrak S_\mathfrak p({\C^n}^*)$,
called
respectively
$A$-Newtonization and Newtonization of the \EA-variety $X$;
see definitions \ref{dfNewtonizationET} and \ref{dfNewtonizationE}.
Let us also recall
that on the space $S({\C^n}^*)$ a symmetric $n$-linear form $\mathfrak p$ is defined;
see definitions \ref{dfPseudo}, \ref{dfPseudoMixed}.
Moreover, the image of the $A$-Newtonization $\nwt(X)$ under the ring homomorphism $\mathfrak S_\vol({\C^n}^*)\to \mathfrak S_\mathfrak p({\C^n} ^*)$
from the corollary \ref{corM1}
is the Newtonization $X^\nwt(X)$ of \EA-variety $X$.
\begin{definition}\label{dfWeakDens}
Let $X$ be an \EA-variety of zero algebraic dimension.
We put $d_w(X)=n!\mathfrak p(X^\nwt)$, and call $d_w(X)$ the \emph{weak density of the \EA-variety} $X$.
\end{definition}
\begin{definition}\label{dfIndex}
Let the sum of the algebraic dimensions of the \EA-varieties $X$ and $Y$ be equal to $n$.
We put $I(X,Y)=n!\:\mathfrak p( X^\nwt\cdot Y^\nwt)$, and call $I(X,Y)$ \emph{intersection index of \EA-varieties} $ X$,$Y$.
\end{definition}
Let us remind that \EA-varieties $X,Y$ of algebraic codimension $k$ are called \emph{numerically equivalent},
if $I(X,Z)=I(Y,Z)$ for any \EA-variety $Z$ of algebraic dimension $k$.
The following three statements are used in the proof of Theorem \ref{thmBase}.
\begin{proposition}\label{prMapTo1}
 The numerical equivalence of \EA-varieties $X$, $Y$ is equivalent to the equality
 $X^\nwt=Y^\nwt$.
\end{proposition}
\begin{proof}
By definition,
the numerical equivalence of \EA-varieties $X,Y$ means that
$\forall Z\colon\,X^\nwt\cdot Z^\nwt=Y^\nwt\cdot Z^\nwt$.
Therefore, the required statement follows from Proposition \ref{prpolRing} (iv)
for $\nu=\mathfrak p$.
\end{proof}
\begin{proposition}\label{prMapTo2}
If $G$-models $X_G$, $Y_G$ of $G$-varieties $X$, $Y$
are numerically equivalent as algebraic varieties in the torus $\T_G$,
then the \EA-varieties $X$, $Y$ are also numerically equivalent.
\end{proposition}
\begin{proof}
According to Theorem \ref{ringEKK},
$\nwt(X_G)=\nwt(Y_G)$.
From the definitions \ref{dfNewtonizationET}, \ref{dfNewtonizationE} it follows that
Newtonizations $X^\nwt$ and $Y^\nwt$ are images of Newtonizations $\nwt(X_G)$ and $\nwt(Y_G)$
under some ring homomorphism $\mathfrak S_\vol(\Z_G)\to\mathfrak S_\mathfrak p({\C^n}^*)$.
Therefore, the required statement follows from Proposition \ref{prMapTo1}.
\end{proof}
The following statement,
is equivalent to the existence of the ring of conditions $\mathcal C(\T)$.
\begin{lemma}\label{lmG}
Let $A,B$ be equidimensional algebraic varieties in $\T$.
Then there exist classes of numerical equivalence of the algebraic varieties $\Pi(A,B)$, $\Sigma(A,B)$ and
algebraic hypersurface $Z(A,B)$
in the torus $\T$,
such,
that the following is true.
If $\tau\in\T\setminus Z(A,B)$,
then

\textbf{(i)} all varieties $A\cap(\tau B)$
are equidimensional and contained in the class $\Pi(A,B)$;

\textbf{(ii)} if $\dim A=\dim B$, then all varieties $A\cup(\tau B)$
are equidimensional and contained in the class $\Sigma(A,B)$;

\textbf{(iii)} the classes $\Pi(A,B)$, $\Sigma(A,B)$ depend only on the numerical equivalence classes containing $A,B$.
\end{lemma}
\noindent
\emph{Proof of Theorem} \ref{thmBase}.
Let us choose the group $G$ so that
that the \EA-varieties $X,Y$ from Theorem \ref{thmBase}
are $G$-varieties.
Let us denote by $A, B$ their $G$-models.
Then, for $\tau=\omega_G(z)$, the algebraic varieties $A\cap(\tau B)$ and $A\cup(\tau B)$
are $G$-models of the \EA-varieties $X\cap(z+Y)$ and $X\cup(z+Y)$.
Therefore, the required statement follows from Lemma \ref{lmG} for $D(X,Y)=\omega_G^{-1}(Z(A,B))$.
\begin{corollary}\label{corNewtHom}
The association of an \EA-variety $X$ with its Newtonization $X^\nwt$ defines an isomorphism between the ring of conditions $\mathcal C(\C^n)$
and the ring $\mathfrak S({\C^n}^*)$.
\end{corollary}
\begin{proof}
Follows from Proposition \ref{prMapTo1} and Theorems  \ref{thmBase}, \ref{thmNewtShift}.
\end{proof}
\noindent
\emph{Proof of properties} \hyperlink{R1}{($\bf R_1$)} -- \hyperlink{R5}{($\bf R_5$)} from \textsection\ref{intro4}.

 \noindent
 \hyperlink{R1}{($\bf R_1$)}
  Follows from the definition of Newtonization and from proposition \ref{prMapTo1}.

 \noindent
\hyperlink{R2}{($\bf R_2$)}
From Proposition \ref{prpolRing} (iii) it follows that
the $\R$-algebra $\mathcal C(\C^n)$ is generated by Newton polyhedra.
On the other hand, if
$\Delta$ -- Newton polyhedron of exponential sum $f$, $t>0$,
then $t\Delta$ is a Newton polyhedron of the exponential sum $g(z)=f(tz)$.
From here the necessary statement follows.

\noindent
 \hyperlink{R3}{($\bf R_3$)}
 A direct consequence of the definitions.

 \noindent
 \hyperlink{R4}{($\bf R_4$)}
 Follows from Proposition \ref{prpolRing} (iv).

 \noindent
 \hyperlink{R5}{($\bf R_5$)}
 Follows from Definitions \ref{dfWeakDens}, \ref{dfIndex}.
 \par\smallskip
 \noindent
 \emph{Proof of Theorem} \ref{thmInd}.
The first statement follows from Proposition \ref{prSdvig1}.
 The second statement is a consequence of Theorem \ref{thmNewtShift}.
 \par\smallskip
 \noindent
\emph{ Proof of Theorem} \ref{thmBKK}.
According to Definition \ref{dfIndex}, the intersection index of \EA-hypersurfaces $X_i$ is equal to
$n! \mathfrak p(\Delta_1\cdot\ldots\cdot\Delta_n)$,
where $\Delta_i$ are the Newton polytopes of hypersurfaces.
By the definition of the ring $\mathfrak S_\mathfrak p({\C^n}^*)$, this equals
 to the mixed pseudovolume of Newton polyhedra $\mathfrak p(\Delta_1\cdot\ldots\cdot\Delta_n)$
 of \EA-hypersurfaces $X_1,\ldots,X_n$.
\section{Strong Density and Intersection Index 
}\label{index}
Here we consider zero-dimensional \EA-varieties and present statements
which describe the properties of the previously defined intersection index of \EA-varieties,
consistent with standard ideas about intersection indices in the context of algebraic geometry.

Let's recall some concepts
 from \textsection\textsection\ref{intro5},\ref{polRing}.

$\bullet$\ $G$ is a finitely generated group in ${\C^n}^*$; $\lambda_1,\ldots,\lambda_N$ is a basis of group $G$
(we assume that the set $\lambda_1,\ldots,\lambda_N$ contains some basis of the space ${\C^n}^*$).
$\omega_G\colon \C^n\to (\C\setminus0)^N$ is a homomorphism of the standard winding,
defined as $z\mapsto(\E^{\lambda_1(z)},\ldots,\E^{\lambda_N(z)})$.
Below we identify the space $\C^n$ with its image $L$ under the mapping
$d\omega\colon\C^n\to\C^N$,
where $d\omega$ is a differential of the winding $\omega_G$,
that is $d\omega(z)=(\lambda_1(z),\ldots,\lambda_N(z))$.
In this case, we fix a Hermitian metric on the space $L$,
taken from the space $\C^n$.
Note that $\omega_G(\C^n)=\exp(L)$.

$\bullet$\ $Y$ is an algebraic variety of codimension $n$ in the torus $(\C\setminus0)^N$.

$\bullet$\ $X=\omega^{-1}_G(Y)$ is an \EA-variety  in $\C^n$.
The variety $Y$ is called the $G$-model of the \EA-variety $X$.
The codimension $Y$ is called the algebraic codimension $\codima(X)$ of the \EA-variety $X$.
By definition, $\dima X=n-\codima X=0$.

$\bullet$\
$\E^AY$ is a shift of the variety $Y$ by the element $\E^A\in(\C\setminus0)^N$,
where $A\in\C^N$.
We denote by
$\E^AX$ the \EA-variety with model $\E^AY$.

$\bullet$\
$\mathcal K_{\rm trop}(Y)$ is the tropical fan of the variety $Y$.
Recall that $\mathcal K_{\rm trop}(Y)$ is a weighted fan in the space $\re\C^N$
of dimension equal to the dimension $Y$.
Note that $\forall A\in\C^N\colon\:\mathcal K_{\rm trop}(\E^AY)=\mathcal K_{\rm trop}(Y)$.

$\bullet$\
$V_K$ is a real subspace in $\C^N$,
generated by the points of the cone $K$.
\par\smallskip

Let $E_G$ be the ring of exponential sums with exponents from the group $G \subset {\C^n}^*$.
Suppose that an \EA-variety $X$ in $\C^n$
is defined by a finite system of equations $\{f_i = 0\}$,
where $f_i \in E_G$.
By choosing a basis $\lambda_1, \ldots, \lambda_N$ of the group $G$,
the exponential sums $f_i$ are uniquely represented
in the form
\[ f_i(z) = P_i\left(\E^{\lambda_1(z)}, \ldots, \E^{\lambda_N(z)}\right), \]
where $P_i$ are Laurent polynomials on the $N$-dimensional torus $(\C \setminus 0)^N$.
The algebraic variety $Y$ in $(\C \setminus 0)^N$,
defined by the equations $\{P_i = 0\}$,
is called the $G$-model of the \EA-variety $X$.
The \EA-variety $X$ is the preimage of $Y$ under the standard windig map
\[ z \mapsto \left(\E^{\lambda_1(z)}, \ldots, \E^{\lambda_N(z)}\right) \]
of $\C^n$ to the torus $(\C \setminus 0)^N$.
By definition,
the algebraic codimension $\codim(X)$ of the \EA-variety $X$ equals $\codim(Y)$;
see \textsection\ref{intro5}.
Therefore, since $\dim(X) = 0$, we have $\codim(Y) = n$.

For any vector $A = (a_1, \ldots, a_N) \in \C^N$, the \EA-variety $\E^AX$ is defined by the system of equations
\[ \{P_i(\E^{a_1} \E^{\lambda_1(z)}, \ldots, \E^{a_N} \E^{\lambda_N(z)}) = 0\}. \]
We consider the \EA-varieties $\E^AX$ as a family of deformations of the \EA-variety $X$ with parameter $A \in \C^N$.
The shift of the $G$-model $Y$ of the \EA-variety $X$ by the element $\left(\E^{-\lambda_1(z)}, \ldots, \E^{-\lambda_N(z)}\right)$ of the torus $(\C \setminus 0)^N$ is the $G$-model of the \EA-variety $\E^AX$.
Therefore, $\dim(\E^AX) = \dim(X) = 0.$
\subsection{Preliminaries}\label{Prel}
\subsubsection{Strong $n$-density}\label{index11}
Let $B_r$ be a ball of radius $r$ centered at zero in the finite-dimensional Euclidean space $E$.
Let $\sigma_n$ denote the volume of an $n$-dimensional ball with radius $1$.
\begin{definition} \label{dfprelim11}
Let $Y \subset E$ be a discrete set of points with nonnegative multiplicities, and let $N(Y, r)$ be the sum of the multiplicities of the points in the set $Y \cap B_r$.
If the limit
\[
\lim_{r \to \infty} \frac{N(Y, r)}{\sigma_n r^n}
\]
exists,
we call it the \emph{strong $n$-density} of the set $Y$ and denote it by $d_n(Y)$.
\end{definition}
\begin{example}\label{exprelim11}
If an \EA-variety $Y \subset \C^1$ is defined by the equation $f(z) = 0$,
then the strong $1$-density $d_1(Y)$ exists and is equal to $\frac{p}{2\pi}$,
where $p$ is the semiperimeter of the Newton polygon of the exponential sum $f$.
In particular,
if $f(z) = \E^{\alpha z} - c$,
then $d_1(Y) = \frac{|\alpha|}{\pi}$
(the perimeter of the polygon ``segment''
is considered to be twice its length).
\end{example}
\begin{proposition}\label{prprelim12}{\rm[cf. Definition \ref{dfPseudo}]}
Let $L$ be a real $n$-dimensional subspace in $\C^n$, and let $L^\bot$ be its orthogonal complement
with respect to the pairing $(*,*) = \re \langle *,* \rangle$.
 Let $\lambda_1, \ldots, \lambda_n$ be a basis of the space $L^\bot$.
Suppose that $\dim_\C L = n$, i.e., $L$ does not contain any nonzero complex subspace in $\C^n$.
Set
\[ Y = \{z \in \C^n \colon \E^{\langle z, \lambda_1 \rangle} = a_1, \ldots, \E^{\langle z, \lambda_n \rangle} = a_n\}, \]
where $0 \ne a_j \in \C$.
Then there exists a lattice $S$ of rank $n$ in the space $L$,
such that for some $z \in \C^n$ we have $Y = z + S$.
Moreover,
\[
d_n(S) = \frac{\cos(L^\bot, i L) \, \vol_n\left(\Pi(\lambda_1, \ldots, \lambda_n)\right)}{(2\pi)^n},
\]
where
$\vol_n\left(\Pi(\lambda_1, \ldots, \lambda_n)\right)$ is the $n$-dimensional volume of the parallelepiped $\Pi(\lambda_1, \ldots, \lambda_n)$ with sides $\lambda_j$, situated in the space $L^\bot$.
\end{proposition}
\noindent
\emph{Proof.}
From the condition $\dim_\C L = n$, it follows that there exists a basis $\mu_1, \ldots, \mu_n$ of the space $L$, such that $\forall p, q \colon (\mu_q, \lambda_p) = i \delta_p^q$.
Then, for any $y \in Y$,
\[ Y = y + 2\pi i (\Z \mu_1 + \ldots + \Z \mu_n). \]
Therefore, $S = 2\pi i (\Z \mu_1 + \ldots + \Z \mu_n)$,
and
\[ d_n(S) = (2\pi)^{-n} \vol_n^{-1}(\Pi(\mu_1, \ldots, \mu_n)). \]
Now, it remains to prove that
\[ \cos(L^\bot, i L) \vol_n(\Pi(\lambda_1, \ldots, \lambda_n)) \vol_n(\Pi(\mu_1, \ldots, \mu_n)) = 1. \]
This follows from the lemma below for $A = i L$, $B = L^\bot$, $C = \C^n$, and $(*,*) = \re \langle *,* \rangle$.
\begin{lemma}\label{lmLA}
Let $A$ and $B$ be $n$-dimensional subspaces of a $(2n)$-dimensional real space $C$
with a scalar product $(*,*)$.
Assume that the pairing $(a \times b) \mapsto (a, b)$ of the spaces $A$ and $B$ is non-degenerate.
Then, if $a_1, \ldots, a_n$ and $b_1, \ldots, b_n$ are mutually dual bases in the spaces $A$ and $B$,
we have
\[ \cos(A, B) \, \vol_n(\Pi(a_1, \ldots, a_n)) \, \vol_n(\Pi(b_1, \ldots, b_n)) = 1. \]
\end{lemma}
\noindent
\emph{Proof.}
Let $p \colon A \to B$ be the orthogonal projection map.
The vectors $p(a_1), \ldots, p(a_n)$ form a basis in the space $B$,
which is dual to the basis $b_1, \ldots, b_n$
with respect to the pairing $\alpha \times \beta \mapsto (p^{-1} \alpha, \beta)$.
It follows that
\[
\vol_n\left(\Pi(p(a_1), \ldots, p(a_n))\right) \, \vol_n\left(\Pi(b_1, \ldots, b_n)\right) = 1.
\]
Now, the desired statement follows from the equality
\[
\vol_n\left(\Pi(p(a_1), \ldots, p(a_n))\right) = \cos(A, B) \, \vol_n\left(\Pi(a_1, \ldots, a_n)\right).
\]
\begin{definition} \label{dfprelim12}
{\rm(1)}
Let $\mathcal{Z} \subset V$ be a lattice in a subspace $V$ of the space $E$, equipped with an integral positive multiplicity $m(\mathcal{Z})$.
A set $Y \subset E$
is called an
$\varepsilon$-perturbation of the shifted
lattice $e + \mathcal{Z}$ if
{\rm(A)} $Y$ belongs to the $\varepsilon$-neighborhood of the shifted lattice,
and {\rm(B)} in the $\varepsilon$-neighborhood of any point $x \in e + \mathcal{Z}$
there are exactly $m(\mathcal{Z})$ points of $Y$.

{\rm(2)}
If the sets $Y_1, \ldots, Y_m$ are $\varepsilon$-perturbations
of the shifted lattices $z_j + \mathcal{Z}_j$,
then the set $\bigcup_{1 \leq j \leq m} Y_j$
is called an $\varepsilon$-perturbation of the union of the shifted lattices
$\bigcup_{1 \leq j \leq m} (z_j + \mathcal{Z}_j)$.
\end{definition}
\begin{lemma}\label{lmprelim111}
Suppose that

$\bullet$ $V$ is an $n$-dimensional subspace in $E$,

$\bullet$ $K$ is a convex $n$-dimensional cone in $V$,

$\bullet$ $K_{R,S}$ is the set of points in $E$ that are at a distance $<R$ from $K$
and at a distance $>S$ from $\partial K$,

$\bullet$ $\mathcal{Z}_1, \mathcal{Z}_2, \ldots$ is a finite set of $n$-dimensional lattices in $V$,

$\bullet$ $Y$ is an $\varepsilon$-perturbation of the union of shifted lattices $\{e_j + \mathcal{Z}_j\}$.

Then, if $R$ is sufficiently large, the strong density $d_n(Y \cap K_{R,S})$ exists and is equal to
$A(K) \sum_j d_n(\mathcal{Z}_j)$,
where $A(K)$ is the angle of the cone $K$ (the angle of the full cone is taken as one).
\end{lemma}
\noindent
\emph{Proof.}
Follows from Definitions \ref{dfprelim12} (1) and \ref{dfprelim12} (2).
\subsubsection{Domains of Relatevely Full Measure}\label{index12}
Let $\mathfrak I=\{I\}$ be a finite set of proper real subspaces
 in vector space $E$.
Let $B_{\mathfrak I}=\bigcup_{I\in \mathfrak I}I$.
For $0<r\in\R$, let $B_{\mathfrak I,r}$ denote the set of points $E$,
located at a distance $\geq r$ from $B_\mathfrak I$.

\begin{definition} \label{dfdRelFull}
Suppose for a domain $U\subset E$,
there exist $\mathfrak I$ and $r>0$,
such that $U\supset B_{\mathfrak I,r}$.
Then the domain $U$ is called a \emph{domain of relative full measure}.
\end{definition}
\noindent
Hereafter, for brevity, we use the abbreviation \rf-domain.
\begin{corollary}\label{cordRelFull}

{\rm(1)}
The union and intersection of a finite number of \rf-domains are \rf-domains.

{\rm(2)}
The property of being an \rf-domain
does not depend on the choice of metric in the space $E$.

{\rm(3)}
If $U\supset B_{\mathfrak I,r}$, and
the subspace $L$ in $E$ does not belong to $B_\mathfrak I$,
then the domain $U\cap L$ remains an \rf-domain in the space $L$.

{\rm(4)}
Let $(E\setminus B_\mathfrak I)^{1},(E\setminus B_\mathfrak I)^{2},\ldots$ be the connected components of the domain $E\setminus B_\mathfrak I$.
Then the domains
$(E\setminus B_\mathfrak I)^{i}\cap B_{\mathfrak I,r}$
are the connected components
of the \rf-domain $B_{\mathfrak I,r}$.
\end{corollary}
Now we will define
the set of real subspaces $\mathfrak I(Y)$ in $\C^N$
and the corresponding \rf-domains,
related to the algebraic variety $Y$ in the torus $(\C\setminus0)^N$.
\begin{definition}\label{df_I}
For $K\in\mathcal K_{\rm trop}(Y)$ (for the notation $\mathcal K_{\rm trop}(Y)$ see \textsection\ref{trop}),
let $V_{K,L}$ denote the minimal real subspace in $\C^N$
containing the union
$K\cup\im\C^N\cup L$.
If  $\dim V_{K,L}<2N$,
then we call the cone $K$ and the subspace $V_{K,L}$
\emph{dangerous} for $Y$.
Let $\mathfrak I(Y)$ be the set of subspaces
dangerous for $Y$,
and define $B_{\mathfrak I}(Y)=\cup_{H\in \mathfrak I(Y)} H$, $B_{\mathfrak I,r}(Y)=B_{\mathfrak I(Y),r}$.
\end{definition}
\begin{corollary}\label{corDangDop}
Let $\mathcal L$ be a fan in $\C^N$,
consisting of a single cone $L$.
Then any point in $\C^N\setminus B_\mathfrak I(Y)$
is admissible for the pair of fans $(\mathcal L,\mathcal K^\C_{\rm trop}(Y))$.
\end{corollary}
\begin{proof}
This follows from the definition \ref{dfTransvCone}.
\end{proof}
The set $\mathfrak I(Y)$ depends on the choice of the tropical fan $\mathcal K_{\rm trop}(Y)$ of the variety $Y$.
When transitioning to a subdivision of the tropical fan,
the set $\mathfrak I(Y)$ may increase.
The following statement is a direct consequence of definition \ref{df_I}.
\begin{corollary}\label{corJC}
Let $Y\subset Z$.
Suppose the fan $\mathcal K_{\rm trop}(Y)$ is a subfan of the fan $\mathcal K_{\rm trop}(Z)$.
Then $\mathfrak I(Y)\subset\mathfrak I(Z)$.
\end{corollary}
Next, we constantly use the notation $\mathcal K^\C_{\rm trop}(Y)$ for a $(2N-n)$-dimensional fan in $\C^N$,
consisting of cones of the form $K+\im\C^N$,
where $K\in\mathcal K_{\rm trop}(Y)$.
Below, we will equip the fan $\mathcal K^\C_{\rm trop}(Y)$ with the structure of a tropical fan in $\C^N$.
To do this, we recall that the Newton polytope $\nwt(Y)$ of the variety $Y$
is an element of the ring $\mathfrak S_\vol(\re {\C^N}^*)$;
see theorem \ref{ringEKK}.
Let $\nwt^\C(Y)$ denote the image of $\iota_*(\nwt(Y))$ in the ring $\mathfrak S_\vol({\C^N}^*)$,
where $\iota\colon\re {\C^N}^*\to{\C^N}^*$ is the embedding of vector spaces.
\begin{corollary}\label{corJC1}
The fan consisting of cones of the tropical fan $\nwt^\C(Y)$ is $\mathcal K^\C_{\rm trop}(Y)$.
\end{corollary}
Let us recall the concept of stable intersection of sets and describe its connection with the definitions of sets
$B_{\mathfrak I}(Y)$ and  $B_{\mathfrak I,r}(Y)$.
\begin{definition}\label{dfStable}
For sets $A,B$ in $\R^M$, a point $e\in A\cap B$ is called stable
if any neighborhood of $e$ intersects the set $A\cap(B+\varepsilon)$ for any sufficiently small $\varepsilon\in\R^M$.
We denote by $A\cap^{\rm st}B$
the set of stable points in $A\cap B$.
This set
is called a \emph{stable intersection} of $A$ and $B$.
The set $A\cap B\setminus (A\cap^{\rm st}B)$ is called \emph{the unstable part of the intersection} of sets $A$ and $B$.
\end{definition}
\begin{corollary}\label{corStable}
If $e\not\in B_{\mathfrak I}(Y)$,
then the following holds.

{\rm(1)}
All points in the unstable part
 of the intersection $L\cap \supp \mathcal K^\C_{\rm trop}(Y)$
disappear when transitioning to
$(L+e)\cap \supp \mathcal K^\C_{\rm trop}(Y)$. As $t\to0$, the set $(L+te)\cap \supp \mathcal K^\C_{\rm trop}(Y)$ approaches the set $ L\cap^{\rm st}\supp \mathcal K^\C_{\rm trop}(Y)$.

{\rm(2)}
As $t\to\infty$,
the distance from $(L+te)\cap\supp \mathcal K^\C_{\rm trop}(Y)$
to the unstable part of the intersection $L\cap \supp \mathcal K^\C_{\rm trop}(Y)$
approaches infinity.
\end{corollary}
\begin{definition}\label{dfStableOfFans}
Let $\mathcal K,\mathcal N$ be fans of cones in the space $\R^N$.
It is true that the stable intersection $(\supp\mathcal K)\cap^{\rm st}(\supp\mathcal N)$
is the support of some subfan in the intersection fan $\mathcal K\cap\mathcal N$;
see \cite{K03}.
We denote this subfan by $\mathcal K\cap^{\rm st}\mathcal N$
and call it the \emph{stable intersection of fans} $\mathcal K,\mathcal N$.
\end{definition}
\subsubsection{Fan of Cones of a Zero-Dimensional \EA-Variety}\label{tropEA}
We describe here in more detail, compared to the definition \ref{dfNewtonizationETtrop}, the construction of the tropical fan $\mathcal E$ of a zero-dimensional \EA-variety $X$ in $\C^n$. Denote by $\mathcal E$ the stable intersection of the fans $\mathcal L \cap^{\rm st} \mathcal K^\C_{\rm trop}(Y)$, where the fan $\mathcal L$ consists of a single cone $L$. It follows from Corollary \ref{corDangDop} that if $A \not\in B_{\mathfrak I}(Y)$, then $\mathcal E = \mathcal L \cap^A \mathcal K^\C_{\rm trop}(Y)$; see Proposition \ref{prConnection}. Recall that, according to Theorem \ref{thmBKK_Trop1}, the fan $\mathcal L \cap^A \mathcal K^\C_{\rm trop}(Y)$ is tropical. Below we provide a detailed description of this tropical fan.
\begin{definition}\label{dfEAFan}
Suppose that for $K \in \mathcal K_{\rm trop}(Y)$ of dimension $N-n$, the dimension of the cone $P = L \cap (K + \im \C^N)$ is $n$. Let $A \not\in B_{\mathfrak I}(Y)$. Then, if $(-A + L) \cap (K + \im \C^N) \ne \emptyset$, we say that the cone $K$ is an $(A-P)$-cone.
\end{definition}
\begin{corollary}\label{corStInt}
Let $A \not\in B_{\mathfrak I}(Y)$. Then (see definition \ref{dfStable})
$$L \cap^{\rm st} \supp \mathcal K_{\rm trop}(Y) = \bigcup_{K \in S} L \cap (K + \im \C^N),$$
where $S$ is the set of all $(A-P)$-cones in $\mathcal K_{\rm trop}(Y)$.
\end{corollary}

\begin{example}\label{exRank}
Let $L = M + iM$, where $M \subset \re \C^N$, $\dim M = n$. In this case, the \EA-variety $X$ with the model $Y$ is quasi-algebraic. Then the fan $\mathcal E$ consists of a single cone $P = iM$. Moreover, a cone $K \in \mathcal K_{\rm trop}(Y)$ of codimension $n$ is an $(A-P)$-cone if and only if $K \cap (-\re A + M) \ne \emptyset$.
\end{example}
\begin{definition}\label{dfEAweight}
Define the weight $w(P)$ of an $n$-dimensional cone $P\in\mathcal E$ as follows.
First, fix some $A\not\in B_{\mathfrak I}(Y)$.
Represent the weight $w(K)$ of the $(A-P)$-cone $K$ as an element
$\Delta_K^n$ in the graded ring $\mathfrak S_\vol(\re{\C^N}^*)$,
where $\Delta_K$ is a full-dimensional,
i.e., $n$-dimensional, convex polytope in the space $\mathfrak S_\vol(V^\bot_K)$; see definition \ref{dfWeightedFan}. 
Consider the image $\pi(\Delta_K)$ of the polytope $\Delta_K$ under the projection map $\pi\colon {\C^N}^*\to L^*$,
where $\pi$ is the linear operator adjoint to the differential of the standard torus wrapping $d\omega\colon L\to\C^N$. Set $w_A(P,K)=\vol_n(\pi(\Delta_K))$,
where $\vol_n(\pi(\Delta_K))$ is the volume of the polytope $\pi(\Delta_K)$,
measured using the Hermitian metric in $L$. Define the weight of the cone $P$ as
$$
w(P)=\sum w_A(P,K),
$$
where the summation is over all $(A-P)$-cones $K$.
\end{definition}
\begin{proposition}\label{prFan_C}
For all $A \not\in B_{\mathfrak I}(Y)$, 
the weights of the cones of the fan $\mathcal E$, defined above, coincide.
\end{proposition}
\begin{proof}
According to Corollary \ref{corDangDop}, any point $A \not\in B_{\mathfrak I}(Y)$ is admissible for the pair of tropical fans $(\mathcal L, \mathcal K^\C_{\rm trop}(Y))$. Therefore, the required statement follows from Theorem \ref{thmBKK_Trop1}.
\end{proof}
We further call $\mathcal E$ the tropical fan of the \EA-variety $X$.
For example, if the \EA-variety $X$ in $\C^1$ is given by the equation $f=0$,
then its tropical fan consists of rays formed by the exterior normals of the sides of the Newton polygon of the exponential sum $f$.
The weights of these rays are equal to the lengths of the corresponding sides.
\begin{definition}\label{EA_nondegen}
A cone $K \subset \C^n$ of dimension $\leq n$ is called \emph{complex non-degenerate} if the dimension of the complex vector subspace generated by the cone is equal to $\dim K$. A fan of cones of dimension $n$ is called complex non-degenerate if each of its cones is complex non-degenerate. A zero-dimensional \EA-variety $X$ is called complex non-degenerate if the fan $\mathcal E$ is complex non-degenerate.
\end{definition}
\begin{corollary}\label{corA-P}
Let an $n$-dimensional cone $P \in \mathcal E$ be complex non-degenerate. Then for any $(A-P)$-cone $K$, it is true that $(V_K + iV_K) \cap L = 0$. Conversely, if for some $(A-P)$-cone $K$ it is true that $(V_K + iV_K) \cap L = 0$, then the cone $P \in \mathcal E$ is complex non-degenerate.
\end{corollary}
\begin{proof}
The subspace $V_K + iV_K$ is the maximal complex subspace in the space $V_K + \im \C^N$. From this follows the required statement.
\end{proof}
\begin{remark}
A tropical fan of an algebraic subvariety in $Y$ can be considered, possibly after transitioning to some subdivisions of the fans, as a subfan of the fan $\mathcal K_{\rm trop}(Y)$; see property \hyperlink{tr_3}{($\bf tr_3$)} in \textsection\ref{trop}. For \EA-varieties, it is not true that the fan of an \EA-subvariety in $X$ is a subfan of the fan $\mathcal E$.
\end{remark}
\subsection{Main Results}
Let $\mathcal E_{S,\R}$ denote the set of points in $\C^n$ that are at a distance
$\geq S$ from the union of complex generated cones of the fan $\mathcal E$; see Definition \ref{EA_nondegen}.
For example, if the \EA-variety $X$ is complex non-degenerate, then $\mathcal E_{S,\R}$ consists of points at a distance $\geq S$ from the support of the $(n-1)$-dimensional skeleton of the fan $\mathcal E$.
\begin{theorem}\label{thmIsolated}

 {\rm(1)} Let $U_X$ be the set of points $A \in \C^N$ such that all points of the \EA-variety $\E^AX$ are isolated. Then the measure of the set $\C^N \setminus U_X$ is zero.
\par\smallskip
In the following statements {\rm (2), (3)} we assume that $A \in B_{\mathfrak I,r}$, where $r$ is sufficiently large.
\par\smallskip
 {\rm(2)}
The set of points of the \EA-variety $\E^AX$ is located at a finite distance from the support of the fan $\mathcal E$.

 {\rm(3)}
For some $S$, the strong density $d_n(\E^AX \cap \mathcal E_{S,\R})$ exists and is equal to the weak density $d_w(X)$.
\end{theorem}
For a complex non-degenerate \EA-variety $X$, Theorem \ref{thmIsolated} can be strengthened as follows.
\begin{theorem}\label{thmStrong}
Let $A \in B_{\mathfrak I,r}$, where $r$ is sufficiently large. Then

{\rm(1)}
the set of points of the \EA-variety $\E^AX$ is discrete.

{\rm(2)}
the strong density $d_n(\E^AX)$ exists and is equal to $d_w(X)$.
\end{theorem}

If a zero-dimensional \EA-variety $X$ is quasi-algebraic, then the fan $\mathcal E$ consists of a single cone $\im \C^n$; see Corollary \ref{corEtropQuasi}. In this case, $X$ is complex non-degenerate, and the set $\mathfrak I(Y)$ consists of subspaces of the form $V + \im \C^N$, where $V$ is a proper subspace in $\re \C^N$. The quasi-algebraic refinement of Theorem \ref{thmStrong} is as follows; see \cite{K22}.
\begin{theorem}\label{thmq1}
{\rm(1)}
The strong density $d_n(\E^AX)$ exists and is equal to $d_w(X)$.

{\rm(2)}
Denote by $\mathfrak C_1, \mathfrak C_2, \ldots, \mathfrak C_M$ the connected components
of the domain $\C^N \setminus B_{\mathfrak I}$. For each of the components $\mathfrak C_i$,
there exists a finite set of $n$-dimensional lattices, equipped with positive integer multiplicities,
$$
\{\mathcal L_{i,j} \subset \im \C^n \colon j = 1, \ldots, N_i\},
$$
such that the following holds. For any $\varepsilon$, there exists $r$ such that for any $(i \leq M, A \in (\C^N \setminus B_{\mathfrak I,r}) \cap \mathfrak C_i)$, the \EA-variety $\E^AX$ is an $\varepsilon$-perturbation of the union of shifted lattices
$$z_1(A) + \mathcal L_{i,1}, \ldots, z_{N_i}(A) + \mathcal L_{i,N_i},$$
where the functions $z_j(A)$ are continuous. The strong density $d_n(\E^AX)$ exists and is equal to $d_w(X)$.
\end{theorem}
\begin{remark}
(1)
From definition \ref{dfNewtonizationET} it follows that the strong density of a quasi-algebraic \EA-variety is equal to the volume of its $A$-Newtonization multiplied by $n!/(2\pi)^n$.

(2)
Some refinements of Theorems \ref{thmStrong} and \ref{thmIsolated} are provided in their proofs in \textsection\ref{pr1}. In particular, statement (2) of Theorem \ref{thmStrong} remains true under certain relaxations of the complex non-degeneracy condition.

(3)
It seems plausible that this statement also remains true for any \EA-varieties, i.e., without the condition of their complex non-degeneracy.
\end{remark}
Let $X, Z$ be \EA-varieties in $\C^n$ of complementary dimensions, i.e., such that $\dima X + \dima Y = n$. Denote by $Y, W$ the models of the \EA-varieties $X, Z$ in the torus $(\C \setminus 0)^N$.
\begin{proposition}\label{prMainIndex}
There exists a finite set of real subspaces $\mathfrak I$ in $\C^N$ such that the following holds. If $A \notin B_{\mathfrak I,r}$, where $r$ is sufficiently large, then $\codim (\E^AY \cap W) = n$.
\end{proposition}
\begin{proof}
This is a standard consequence of the properties of the tropicalizations of algebraic varieties; see Corollary \ref{corAlgTrop}.
\end{proof}
Applying the above theorems to a zero-dimensional \EA-variety with the model $\E^AY \cap W$, we obtain statements that interpret the notion of the intersection index of \EA-varieties through the strong densities of their intersections.
\section{Approximations and Atypical Components}\label{tropTor}

This section presents statements used in the proofs of theorems on the strong density of \EA-varieties from \textsection\ref{index}.

$\bullet$\ In \textsection\ref{trop}, some properties of tropicalizations and toric completions of algebraic varieties are listed.

$\bullet$\ In \textsection\ref{Tongue}, $\varepsilon$-approximations of algebraic varieties in $(\C \setminus 0)^N$ using a finite set of toric tongues are described. This theorem is a consequence of the theory of toric varieties, as well as the properties of the tropicalization of an algebraic variety given in \textsection\ref{trop}. At the end of \textsection\ref{Tongue}, a brief outline of the proof of the theorem on tongues is provided.

$\bullet$\ In \textsection\ref{atypical}, some properties of atypical components of \EA-varieties are derived.
\subsection{Tropicalizations of Algebraic Varieties}\label{trop}
Recall that the Newton polyhedra of Laurent polynomials on the torus $(\C\setminus 0)^N$ are
convex polyhedra in the character space $\re{\C^N}^*$.
The vertices of the Newton polyhedra are points of the character lattice $\Z^N \subset \re{\C^N}^*$ of the torus $(\C\setminus 0)^N$.
The Newtonization $\nwt(Y)$ of an algebraic variety $Y \subset (\C\setminus 0)^N$ is an element
of the subring $\mathfrak{S}_\vol(\Z^N)$ generated by the Newton polyhedra
in the ring $\mathfrak{S}_\vol(\re{\C^N}^*)$; see \textsection\ref{AlgNewt}.
\begin{definition}\label{defTropAlg}
The tropicalization of the element $\nwt(Y)$ (see definition \ref{dfEquiWeightedTrop}) is called the \emph{tropicalization of the algebraic variety} $Y$.
\end{definition}

The tropicalization of the variety $Y$ is, up to subdivision,
a weighted fan of cones of dimension $\dim Y$ in the space $\re(\C^N)$; see definition \ref{dfWeightedFan}.
Any of these fans is denoted by $\mathcal{K}_{\rm trop}(Y)$ and called the tropical fan of the variety $Y$.
For any pair of tropical fans of the variety,
there exists a tropical fan that is their common subdivision.
The weights of the cones in these refinements are inherited.

In the space $\re(\C^N)$, there is a lattice $\Z^N$ associated with the one-parameter subgroups of the torus $(\C\setminus 0)^N$. All cones of the fan $\mathcal{K}_{\rm trop}(Y)$ can be chosen so that for all $K \in \mathcal{K}_{\rm trop}(Y)$, the subspace $V_K$ generated by the points of the cone $K$ is rational, i.e., generated by vectors from the lattice $\Z^N$. Denote by $V_K^\bot$ the character space of the quotient torus $(\C\setminus 0)^N / \T_K$, where $\T_K$ is the subtorus in $(\C\setminus 0)^N$ generated by the exponents of the points of the cone $K$. We further identify the space $V_K^\bot$ with the orthogonal complement of $V_K$ in the space $\re(\C^N)$. Consider in $V_K^\bot$ a volume form such that the volume of the fundamental cube of the character lattice is equal to one. Using this volume form, we consider the weight $w(K)$ as the number $\vol_{N-k}(w(K))$; see remark \ref{rmOmegaNumber}. From the Kushnirenko-Bernstein-Khovanskii theorem, it follows that the number $(N-k)! w(K)$ is an integer and positive. Such tropical varieties are called algebraic. Algebraic tropical varieties form a ring with operations defined in \textsection\ref{tropS}.

Applying the definition of the ring of conditions of the torus and theorems \ref{ringEKK} and \ref{thmBKK_Trop1}, we obtain the following statement.

\begin{corollary}\label{corAlgTrop}
{\rm(1)} The assignment of an algebraic variety $Y$ to its tropicalization extends to an isomorphism of the ring of conditions of the torus $(\C\setminus 0)^N$ onto the ring of algebraic tropical varieties in the space $\re(\C^N)$.

{\rm(2)} For algebraic varieties $X,Y$ in the torus $(\C\setminus 0)^N$, there exists an algebraic hypersurface $D(X,Y)$ in $(\C\setminus 0)^N$ such that for $g \notin D(X,Y)$, the following holds: the tropicalization of the variety $X \cap (gY)$ is equal to the product of the tropicalizations of $X$ and $Y$ in the ring of algebraic tropical varieties.
\end{corollary}

Let $Y_{\rm trop}$ be a completion of $(\C\setminus 0)^N$ associated with the fan $\mathcal{K}_{\rm trop}(Y)$. On the algebraic variety $Y_{\rm trop}$, the action of the torus $(\C\setminus 0)^N$ is defined. The orbits of this action are in one-to-one correspondence with the cones of the fan $\mathcal{K}_{\rm trop}(Y)$. More precisely, as a set with a torus action, $Y_{\rm trop}$ coincides with the union of the quotient tori $(\C\setminus 0)^N$ over the subtori $\T_K$ generated by the exponents of the elements of the cones $K \in \mathcal{K}_{\rm trop}(Y)$. The orbit corresponding to the cone $K$ can be identified with the quotient torus $(\C\setminus 0)^N / \T_K$. Further, we use references to statements about the properties of toric varieties from the following list.

\hypertarget{tr_1}${(\bf tr_1)}$ The closure of the variety $Y$ in $Y_{\rm trop}$ is compact.

\hypertarget{tr_2}${(\bf tr_2)}$ If the closure of $Y$ in some toric completion of the torus $(\C\setminus 0)^N$ is compact, then the support of the fan associated with this completion contains $\supp(\mathcal{K}_{\rm trop}(Y))$.

\hypertarget{tr_3}${(\bf tr_3)}$
The set of limit points $Y^\infty(K)$ of the variety $Y$
on the toric orbit $(\C\setminus 0)^N / \T_K$ is an algebraic variety in the torus $(\C\setminus 0)^N / \T_K$
of codimension equal to $\codim Y$.
The tropical fan of the variety $Y^\infty(K)$ is the $V_K$-factorization of the fan $\mathcal{K}_{\rm trop}(Y)$;
see definition \ref{dfFact}.

\hypertarget{tr_4}${(\bf tr_4)}$ The number of limit points of $Y$ on the toric orbit of the variety $Y_{\rm trop}$ corresponding to the cone of maximum dimension $K$ in the fan $\mathcal{K}_{\rm trop}(Y)$ is equal to $p! w(K)$, where $p = \codim Y$.

\hypertarget{tr_5}${(\bf tr_5)}$ The set $\re\log Y$ in the space $\re(\C^N)$ is located at a finite distance from the support of the fan $\mathcal{K}_{\rm trop}(Y)$. Conversely, if $\re\log Y$ is located at a finite distance from the support of some $k$-dimensional fan $\mathcal{L}$, then $\supp(\mathcal{K}_{\rm trop}(Y)) \subset \supp(\mathcal{L})$. In other words, the \emph{distance between the sets} $\re\log Y$ and $\supp(\mathcal{K}_{\rm trop}(Y))$ is finite. Recall that the set $\re\log Y$ is called the \emph{amoeba of the algebraic variety} $Y$. Note that this property of tropicalizations is a consequence of theorem \ref{thmApproximation} from \textsection\ref{Tongue}.

\hypertarget{tr_6}${(\bf tr_6)}$ If $Z$ is an algebraic subvariety of $Y$, then the tropical fans can be chosen so that $\mathcal{K}_{\rm trop}(Z)$ is a subfan in the fan $\mathcal{K}_{\rm trop}(Y)$.

\hypertarget{tr_7}${(\bf tr_7)}$
Let $\mathfrak T_K$ be the Lie algebra of the quotient torus $(\C\setminus0)^N/\T_K$.
Then the space $\re \mathfrak T_K$ can be identified with
the quotient space $\re \C^N/V_K$,
where $V_K$ is the subspace in $\re\C^N$ generated by the points of the cone $K$.
Under this identification,
the $V_K$-factorization of the fan ${\mathcal K}_{\rm trop}(Y)$
(see definition \ref{dfFact})
becomes the tropical fan of the variety $Y^\infty(K)$ in $(\C\setminus0)^N/\T_K$,
consisting of the limit points of $Y$ on the toric orbit $(\C\setminus0)^N/\T_K$.

\hypertarget{tr_8}${(\bf tr_8)}$\
Let $\L$ and $\mathcal L$ be respectively a $1$-dimensional subtorus in $(\C\setminus0)^N$ and it's Lie algebra,
and let
$\pi\colon (\C\setminus0)^N \to (\C\setminus0)^N/\L$  be a projection mapping.
Then the image of $\supp(\mathcal K_{\rm trop}(Y))$ under the projection
$d\pi\colon\re\C^N\to \re\C^N/\re(\mathcal L\cap \C^N)$
coincides with the support of the tropical fan of the variety $\pi(Y)$.

\hypertarget{tr_9}${(\bf tr_9)}$\
If the variety $Y$ is irreducible,
then any cone $K\in\mathcal K_{\rm trop}(Y)$ is a face of one of the fan cones of maximum dimension.

\par\smallskip
The listed properties are roughly described by the statement from \cite[Part 5]{EKK},
called the “good compactification theorem”;
see also \cite{K03}.
\subsection{Toric Tongues of Algebraic Varieties}\label{Tongue}
\begin{definition} \label{dfPert}
Let

\noindent
$\bullet$\
$\V$ be a subtorus of $\T$ of dimension $N-k$

\noindent
$\bullet$\
$\V_\varepsilon$ be an $\varepsilon$-neighborhood of the identity in the torus $\V$

\noindent
$\bullet$\
$\L$ be a subtorus of $\T$ of dimension $k$ such that $\#(\L\cap\V)=1$

\noindent
$\bullet$\
 $U$ be an open domain in the shifted subtorus $\tau\L$

\noindent
$\bullet$\
$U_\varepsilon=\{l\cdot v\in\T\:\vert\: l\in U, v\in \V_\varepsilon\}$

\noindent
$\bullet$\
$\pi_\varepsilon\colon U_\varepsilon\to U$ be the mapping defined as $\pi_\varepsilon\colon l\cdot v\mapsto l$

\noindent
$\bullet$\
$Y$ be a $k$-dimensional algebraic variety in the $N$-dimensional torus $\T$.

We call $Y\cap U_\varepsilon$
an $\varepsilon$-\emph{perturbation of the domain} $U$,
if the restriction
$\pi_\varepsilon\colon Y\cap U_\varepsilon\to U$
is a finite-sheeted covering of the domain $U$.
It is assumed that locally at any point of $Y\cap U_\varepsilon$,
the sheet of the covering $\pi_\varepsilon$ may have multiplicity not equal to one.
\end{definition}
\begin{definition} \label{dfTonque} Let

\noindent
$\bullet$\
$K$ be a $k$-dimensional convex polyhedral rational cone in the space of one-parameter subgroups $\re\mathfrak T=i\im\mathfrak T$.

\noindent
$\bullet$\
$V_K$ be a $k$-dimensional subspace in $\re\mathfrak T$ generated by the cone $K$.

\noindent
$\bullet$\
$\T_K$ be the subtorus of $\T$ generated by the exponents of the subspace $V_K$.

The domain $t_{K,\tau}=\tau\exp(\sqrt{-1}V_K)\exp(K)$ in the shifted subtorus $\tau\T_K$ 
is called a
\emph{toric tongue with base $K$ and shift} $\tau\in\T$.
\end{definition}
Hereafter, we use the following notation.
\par\smallskip
\noindent
$\bullet$\
$\mathcal K$ be a $k$-dimensional fan of convex rational polyhedral cones in the 
$\re\mathfrak T$.

\noindent
$\bullet$\
$\supp\mathcal K$ be the union of the cones of the fan $\mathcal K$.

\noindent
$\bullet$\
$\mathcal K^m$ be the subfan of the fan $\mathcal K$, consisting of cones of dimension $\leq k-m$.

\noindent
$\bullet$\
$A_R$ be the $R$-neighborhood of the set $A\subset\re\mathfrak T$, i.e., the set of points located at a distance $\leq R$ from $A$.

\noindent
$\bullet$\
$O_R(\mathcal K)=\{\tau\in\T\colon{\rm Re}\log\tau\not\in (\supp\mathcal K^1)_R\}$.
\begin{definition} \label{dfApproxSet}
A finite set of pairwise disjoint toric tongues $T_\varepsilon(Y)$ is called an \emph{$\varepsilon$-approximating set of tongues of the variety $Y$ with an approximating tropical fan} $\mathcal K$, if the following conditions hold.

{\rm(i)} Each of the tongues
$t_{K,\tau}\in T_\varepsilon(Y)$ has an $\varepsilon$-perturbation $t_{K,\tau,\varepsilon}\subset Y$.
The degree of the covering $\pi_\varepsilon\colon t_{K,\tau,\varepsilon}\to t_{K,\tau}$ is denoted by $m(t_{K,\tau})$.
We call $m(t_{K,\tau})$ the \emph{weight of the toric tongue $t_{K,\tau}$}.
The sum of the weights of all tongues with base $K$ equals $(N-k)!\:w(K)$,
where $w(K)$ is the weight of the cone $K\in\mathcal K$.

{\rm(ii)} The set of bases of the tongues $t_{K,\tau}\in T_\varepsilon(Y)$ is the set of
$k$-dimensional cones of the $k$-dimensional fan of cones $\mathcal K$.

{\rm(iii)} If $R$ is sufficiently large, then
$$
Y\cap O_{R}(\mathcal K)=\left(\bigcup\nolimits_{t_{K,\tau}\in T_\varepsilon(Y)}t_{K,\tau,\varepsilon}\right)\bigcap O_{R}(\mathcal K).
$$

{\rm(iv)} The coverings $\pi_\varepsilon\colon t_{K,\tau,\varepsilon}\to t_{K,\tau}$ are unramified.
\end{definition}
\begin{theorem}\label{thmApproximation}
For some tropical fan ${\mathcal K}_{\rm trop}(Y)$,
an $\varepsilon$-approximating set of tongues $T_\varepsilon(Y)$ with an approximating fan ${\mathcal K}_{\rm trop}(Y)$ exists
for any sufficiently small $\varepsilon$.
\end{theorem}
We provide a brief outline of the proof of the theorem.

Recall that there is a one-to-one correspondence $K \Leftrightarrow O_K$
between the cones $K\in{\mathcal K}_{\rm trop}(Y)$ and
the orbits $O_K$ of the torus $\T$ in the toric variety $Y_{\rm trop}$.
Namely, the orbit $O_K$ can be identified with the quotient torus $\T/\T_K$;
see \textsection\ref{trop}.

If $\dim K=k$, then the set $\bar Y_K=\{t\}$,
consisting of the limit points of the variety $Y$ on the orbit $O_K$,
is a zero-dimensional algebraic variety,
i.e., it is finite;
see \hyperlink{tr_3}{($\bf tr_3$)}, \hyperlink{tr_4}{($\bf tr_4$)}.
For a limit point $t\in\bar Y_K$, choose its preimage $\tau$ under the projection
map $\T\to\T/\T_K$.
The set $t_{K,\tau}$ is a shifted toric tongue with base $K$ and shift $\tau$.
Assign the tongue $t_{K,\tau}$ a multiplicity $m(t_{K,\tau})$,
equal to the multiplicity of the point $t$ in the zero-dimensional algebraic variety $\tilde Y_K$.

Let $R$ be sufficiently large.
For all such tongues $t_{K,\tau}$ for all $K\in{\mathcal K}_{\rm trop}(Y)$,
choose shifts of the tongues $\tau\in\T$ such that for any $k$-dimensional cone $K\in{\mathcal K}_{\rm trop}(Y)$ and any $t\in\tilde Y_K$, \ 1) $\re\log \tau\in K_R$, and\ 2) the distance from $\re\log \tau$ to $\partial K$ is greater than $R$.
Then,
if $R$ is sufficiently large,
the $\varepsilon$-perturbations $t_{K,\tau,\varepsilon}$
of the tongues $t_{K,\tau}$ exist and satisfy conditions (i), (ii), and (iii) of Definition \ref{dfApproxSet}.
In this case,
the degree of the covering $\pi_\varepsilon\colon t_{K,\tau,\varepsilon}\to t_{K,\tau}$ equals the multiplicity of the corresponding point $t$ in the variety $\tilde Y_K$.
However, condition (iv) may not hold,
i.e., the corresponding covering $\pi_\varepsilon$ of the shifted tongue $t_{K,\tau}$ may be ramified.
To obtain the property of unramified coverings,
it is sufficient to refine the construction of the tongues as follows.

Consider the subset in $Y$,
consisting of all branching points of all coverings $\pi_\varepsilon\colon t_{K,\tau,\varepsilon}\to t_{K,\tau}$.
This set belongs to some proper subvariety $X$ in $Y$ of dimension $<k$.
Therefore, the support of the fan ${\mathcal K}_{\rm trop}(X)$ belongs
to the support of the fan ${\mathcal K}_{\rm trop}(Y)$.
Since $\dim {\mathcal K}_{\rm trop}(X)<k$,
there exists a subdivision $\mathcal K'$ of the fan ${\mathcal K}_{\rm trop}(Y)$ such that
${\mathcal K}_{\rm trop}(X)$ is a subfan of its $(k-1)$-dimensional skeleton; see \hyperlink{tr_6}{($\bf tr_6$)}.
The fan $\mathcal K'$ is also a support of the tropicalization ${\rm trop}(X)$ of the variety $X$;
see \textsection\ref{trop}.
Repeating the construction of the set of tongues with
the replacement of the approximating fan ${\mathcal K}_{\rm trop}(Y)$ by the fan $\mathcal K'$,
we obtain the property of coverings $\pi_\varepsilon$ to be unramified.
%
\begin{remark} \label{remTonque}
(1)\
If $\dim Y=0$, then $T_\varepsilon(Y)$
coincides with the set of points of the variety $Y$.

(2)\
Let $\delta<\varepsilon$.
Then, if $\dim Y>0$, the set of tongues $T_\varepsilon(Y)$ is not $\delta$-approximating,
but becomes $\delta$-approximating with suitable changes in the shifts of the tongues from the set $T_\varepsilon(Y)$.
Such shifts of tongues arise with the increase of the parameter $R$ in the above construction.

(3)\
If $\dim Y>0$, the tongues from the set $T_\varepsilon(Y)$
are defined ambiguously.
For example,
any subdivision of the approximating fan
remains approximating.
As the approximating fan is refined,
the toric tongues are also refined.

(4)\
  Under the shift $Y\mapsto gY$
  the toric tongues from the set
  $T_\varepsilon(Y)$ also shift:
   $T_\varepsilon(gY)=\{t_{K,g\tau}\:\vert\:\, t_{K,\tau}\in T_\varepsilon(Y)\}$.
\end{remark}
\subsection{Atypical Components}\label{atypical}
Let $Y$ be an irreducible algebraic variety of dimension $d$
in $(\C\setminus0)^N$,
and $L$ be an $n$-dimensional complex subspace in $\C^N$.
\begin{definition}\label{dfAnom}
The point $\in Y$ is called $L$-atypical,
if, for any neighborhood of zero $U$ in the space $L$,
the dimension of the intersection $\left(y\cdot\exp(U)\right)\cap Y$ is greater than $\max (n+d-N,0)$.
Let $Y^L$ denote the set of $L$-atypical points in $Y$.
Variety $Y$
is called $L$-atypical,
 if $Y=Y^L$.
\end{definition}
\begin{corollary}\label{corAnom1}
Let $X$ be an \EA-variety in $L$,
and $Y$ be a model of the \EA-variety $X$.
Then, if $X$ is not discrete,
$Y^L\ne\emptyset$.
\end{corollary}
\begin{proposition}\label{pratype1}
The set $Y^L$ is an $L$-atypical algebraic variety.
\end{proposition}
\begin{proof}
Choose a basis $z_1,\ldots,z_n$ in $L$,
and define the tangent vector $\xi_i$ at a point $x\in(\C\setminus0)^N$ as
$\xi_i=\lim_{z_i\to0}\frac{x\cdot\exp(z_i)}{z_i}$.
Let $f^1,\cdots,f^q$ be some system of equations of the variety $Y$ in $(\C\setminus0)^N$.
Consider the matrix
$$
A(x)=\{A_{i,j}(x)=
\frac{\partial f^j}{\partial \xi_i}\colon 1\leq j\leq q,1\leq i\leq n\}
$$
The elements of the matrix $A(x)$ are Laurent polynomials in $N$ variables $x_1,\ldots,x_N$.
Therefore, for any $I>0$ the set
$$Y^I_1=\{x\in Y\colon \rank(A)\leq n - I\}$$
is an algebraic variety.
For $i\geq1$ we put $Y^I_{i+1}=(Y^I_i)^I_1$.
Then, starting from some $k$,
for any non-negative $i,j$, $Y^I_{k+i}=Y^I_{k+j}$.
We put $Y^I=Y^I_{k}$.

We will prove that if $I=\max(n+d-N,0)+1$,
then $Y^I=Y^L$.
By construction, $Y^L\subset Y^I$.
Consider the mapping of a small neighborhood $U_y$ of a non-singular point $y$ in $Y^I$
to the set
of connected components of the intersections of $U_y$ with the layers of the foliation $x\cdot\exp(L)$.
Then, by Sard's lemma,
the Hausdorff dimension of the image of $U_y$
does not exceed $\dim (Y^I) - \max(n+\dim(Y^I)-N,0)-1$.
It follows that $U_y\subset Y^L$.
From the closedness of the set $Y^I$, it follows that
the singular points of the variety $Y^I$ are also $L$-atypical.
\end{proof}
\begin{proposition}\label{pratype2}
Let $Y=Y^L$.
Then in $(\C\setminus0)^N$ there exists a subtorus $\A$ with a Lie algebra $A\subset\C^N$,
such that the variety $Y$ is also $A$-atypical.
\end{proposition}
\begin{proof}
If a point $y$ is $A$-atypical,
we will say that the subtorus $\L$ serves the point $y$.
Theorem \ref{thmAnom} states that for any point $y\in Y$ there exists a serving subtorus.
From the irreducibility condition of the variety, it follows that there exists a subtorus
serving all points of $Y$ simultaneously.
\end{proof}
\begin{proposition}\label{pratype3}
If $Y=Y^L$, then for any $d$-dimensional cone $K\in{\mathcal K}_{\rm trop}(Y)$,
it holds that $\dim(L\cap (V_K+iV_K))> 0$.
\end{proposition}
\begin{proof}
Choose a toric tongue $t_{K,\tau}$ from the set of $\varepsilon$-approximating tongues $T_\varepsilon(Z)$;
see Definition \ref{dfApproxSet}.
Let $t_{K,\tau,\varepsilon}$ be an $\varepsilon$-perturbation of the toric tongue $t_{K,\tau}\in T_\varepsilon(Y)$
with base $K$.
It follows from Theorem \ref{thmApproximation} that
in the tangent spaces of points
there are complex lines arbitrarily close to the tangent spaces of some points of the torus $\T_K$.
This implies the desired statement.
\end{proof}
\begin{corollary}\label{corpratype31}
Let $\forall K\in\mathcal K_{\rm trop}(Y)\colon\dim \left(L\cap(V_K+iV_K)\right)=0$.
Then $Y^L=\emptyset$.
\end{corollary}
\begin{proof}
The tropical fan ${\mathcal K}_{\rm trop}(Y^L)$ can be chosen such
that it is a subfan of the fan ${\mathcal K}_{\rm trop}(Y)$.
Therefore, the desired statement follows from Proposition \ref{pratype3}.
\end{proof}
\begin{corollary}\label{df_I2}
For
$K\in{\mathcal K}_{\rm trop}(Y)$,
consider the toric orbit $(\C\setminus0)^N/\T_K$
in the completion of the torus $(\C\setminus0)^N$,
corresponding to the fan ${\mathcal K}_{\rm trop}(Y)$.
Let $\bar Y_K$ be the variety in the torus $(\C\setminus0)^N/\T_K$,
consisting of the limit points
of the variety $Y$,
and $\bar L$ be the image of the subspace $L$ under the differential
of the projection map $(\C\setminus0)^N\to(\C\setminus0)^N/\T_K$.
Then, if the conditions of Corollary \ref{corpratype31} are satisfied for the variety $Y$,
then $(\bar Y_K)^{\bar L}=\emptyset$.
\end{corollary}
\begin{proof}
From property
\hyperlink{tr_3}{($\bf tr_3$)} in \textsection\ref{trop},
it follows
that the conditions of Corollary \ref{corpratype31} are also satisfied for the variety $\bar Y_K$.
\end{proof}
\begin{remark}
The statement
converse to Proposition \ref{pratype3}
seems plausible.
Namely,
if for any $d$-dimensional cone $K\in{\mathcal K}_{\rm trop}(Y)$
it holds that $\dim(L\cap (V_K+iV_K))> 0$,
then $Y=Y^L$.
\end{remark}
\section{Proofs of Theorems from \textsection\ref{index}}\label{pr1}
%
%
%
\subsection{Theorem \ref{thmIsolated}}\label{prPr11}
Recall that $Y$ is a model of the zero-dimensional \EA-variety $X$,
i.e. $\codim(Y)=n$.
\subsubsection{Proof of Theorem \ref{thmIsolated} (1)}\label{prPr111}
\begin{lemma}\label{lm(1)}
Let the variety $Y$ be $L$-atypical (see Definition {\rm\ref{dfAnom}}).
Then

{\rm(1)}
for any $A\in\C^N$, the set of isolated points of the \EA-variety $\E^AX$ is empty

{\rm(2)}
the measure of the set
consisting of points $A\in\C^N$
such that the \EA-variety $\E^AX$ is not empty
is zero

{\rm(3)}
the weak density $d_w(X)$ is zero.
\end{lemma}
\begin{proof}
Under the toric shift $Y\to\E^AY$, the property of $L$-atypicality of $Y$ is preserved.
Therefore, the first statement is a direct consequence of Definition \ref{dfAnom}.

Let $\log Y=\exp^{-1}Y$.
Let $p\colon\log Y\to\C^N/L$ be the restriction of the projection map
$\C^N\to\C^N/L$ to the subvariety $\log Y\subset\C^N$.
The second statement (2) follows from Sard's lemma for the map $p$.

According to Definition \ref{dfWeakDens},
$d_w(X)=n!\mathfrak p(\mathcal E)$,
where $\mathcal E$ is the tropical fan of the \EA-variety $X$.
We will prove that $\mathfrak p(\mathcal E)=0$.
According to Proposition \ref{pratype3},
for each of the cones $K\in\mathcal K_{\rm trop}(Y)$ of dimension $N-n$,
the dimension of the subspace $(V_K+iV_K)\cap L$ in $\C^N$ is non-zero.
Hence, by
Definition \ref{dfEAFan} and Proposition \ref{prFan_C},
for each $n$-dimensional cone $P\in\mathcal E$ the following holds:
the subspace $V_P$ in $L$
(recall that we identify the spaces $L$ and $\C^n$)
contains some non-zero complex subspace in $\C^n$.
Using equality (\ref{eq_pFor}),
we get $d_w(X)=0$,
since all coefficients $c(P)$ from
(\ref{eq_pFor}) are zero.
Thus, we get the desired equality $d_w(X)=0$.
\end{proof}
Now we proceed directly to the proof of Theorem \ref{thmIsolated} (1).
In the case of an $L$-atypical variety $Y$, both statements of Theorem \ref{thmIsolated}
are direct consequences of Lemma \ref{lm(1)}.
The proof of the first statement of Theorem 
for an arbitrary \EA-variety $X$
is as follows.
From the definition of $L$-atypicality, it follows
that the set $U_X$ 
coincides with the set of points $A\in\C^N$
such that 
$\log(Y^L)\cap (-A+L)=\emptyset$.
According to Proposition \ref{pratype1},
the set of $L$-atypical points $Y^L$ of the variety $Y$
is an algebraic variety.
Therefore, the desired statement follows from Lemma \ref{lm(1)} (2).
\subsubsection{Proof of Theorem \ref{thmIsolated} (2)}\label{prPr112}
According to property \hyperlink{tr_5}{($\bf tr_5$)} from \textsection\ref{trop},
the distance from the set $\log Y$ to the set $\supp \mathcal K^\C_{\rm trop}(Y)$ is finite;
recall that $\mathcal K^\C_{\rm trop}(Y)=\{K+\im\C^N\colon K\in \mathcal K_{\rm trop}(Y)\}$.
From Corollary \ref{corStable} it follows that as $r$ increases and $A\in B_{\mathfrak I,r}(Y)$,
the distance from the set $(-A+L)$ to the unstable part of the intersection $L\cap\supp \mathcal K^\C_{\rm trop}(Y)$
increases unboundedly;
see Corollary \ref{corStable}.
On the other hand,
the distance from the set $(-A+L)\cap\log(Y)$ to the set $(-A+L)\cap\mathcal K^\C_{\rm trop}(Y)$
remains bounded by some quantity that does not depend on the choice of $r$.
Therefore, the required statement follows from the description of the fan $\mathcal E$ in \textsection\ref{tropEA}.
\subsubsection{Proof of Theorem \ref{thmIsolated} (3)}\label{prPr113}
Let $\mathcal E$ be the tropical fan of the zero-dimensional \EA-variety $X\subset L$ with the model $Y\subset(\C\setminus0)^N$.
For an $n$-dimensional complex non-degenerate cone $P\in\mathcal E$, denote by
$P_R$ the set of points in $L$ whose distance to $P$ does not exceed $R$.
Denote also by
$P_{R,S}$ the subset of $P_R$ consisting of points at a distance $> S$ from the boundary of the cone $P$.
We use the notation $L_P$ for the real subspace of $L$
generated by the cone $P$, and $L_P^\bot$ for the orthogonal complement of $L_P$ in $L$,
i.e. $\{L_P^\bot=l\in L\colon\: \re\langle l,L_P\rangle=0\}$.
By definition, $\dim L_P=\dim L_P^\bot=n$.
The proof of Theorem \ref{thmIsolated} (3) is based on the calculation of the strong density $d_n(\E^AX\cap P_{R,S})$
for $\dim P=n$;
see Assertion \ref{ass1} below.

The proof of Assertion \ref{ass1} is based on the $\varepsilon$-approximation of the algebraic variety $\E^AY$;
see \textsection\ref{Tongue}.
Namely, we approximate
the variety $\E^AY$
by the union of toric tongues $t_{K,\tau}$ from the set $T_\varepsilon(\E^AY)$,
and replace the part of the \EA-variety $\E^AX\cap P_{R,S}$ by the preimage
of the union of some subset of the union of approximating tongues $t_{K,\tau}$ under the standard winding map $\omega\colon\C^n\to(\C\setminus0)^N$.
We prove that, after such a replacement,
the formula for strong density in Assertion \ref{ass1} becomes valid; see Proposition \ref{prass1_3}.
Next we replace the tongues $t_{K,\tau}$ with their $\varepsilon$-perturbations $t_{K,\tau,\varepsilon}$,
and after that we finish the proof of the statement by referring to Theorem \ref{thmApproximation}.
Below we use the following definitions and assumptions.

\hypertarget{p1}${(\bf p_1)}$\ $P\in\mathcal E$, $\dim P=n$,
the cone $P$ is complex non-degenerate.

\hypertarget{p2}${(\bf p_2)}$\ $\T_K$ is a subtorus in $(\C\setminus0)^N$,
generated by the exponents of the cone $K\in\mathcal K_{\rm trop}(Y)$.
We consider the Lie algebra $\mathfrak T_K$ of the torus $\T_K$ as a subspace of $V_K+iV_K$ in $\C^N$.
The space $i V_K$ is considered as a Lie algebra of the maximal compact subtorus of the torus $\T_K$.

\hypertarget{p3}${(\bf p_3)}$\
$x_1,\ldots,x_N$ are coordinates in the torus $(\C\setminus0)^N$,
such that each of the first $n$ coordinates identically equals $1$ on the torus $\T_K$.
Let
$\mathbb N$ be an $n$-dimensional subtorus in $(\C\setminus0)^N$,
given by the equations $x_{n+1}=\ldots=x_{N}=1$.
Note that the space $V_K^\bot$ defined in \textsection\ref{polRing3} coincides with the character space of the torus $\mathbb N$.
We further consider the Lie algebra $\mathfrak N$ of the torus $\mathbb N$ as a subspace of $V_\mathbb N+i V_\mathbb N$ in $\C^N$.
We also use the notation $\Z_K$
for the lattice in the space $i V_\mathbb N$,
consisting of the preimages of the unit
under the exponential map $\exp\colon\mathfrak N\to\mathbb N$.

\hypertarget{p4}${(\bf p_4)}$\
$\mathcal K_{A-P}$ is the set of $(A-P)$-cones of the fan $\mathcal K_{\rm trop}(Y)$;
see Definition {\rm\ref{dfEAFan}}.
Further, we assume that $K\in\mathcal K_{A-P}$.

\hypertarget{p5}${(\bf p_5)}$\
$L=dw(\C^n)$,
where $dw$ is the differential of the standard winding.
In this case, the space $L$ is endowed with a Hermitian metric $\langle *,*\rangle$ taken from the space $\C^n$.
The cone $P=(K+\im\C^N)\cap L$ is an $n$-dimensional cone of the fan $\mathcal E$.
In this case, $L_P=(V_K+\im\C^N)\cap L$.
Recall also that the space $L$ is assumed to be endowed with a Hermitian metric,
consistent with the Hermitian metric in $\C^n$ via the isomorphism $d\omega\colon\C^n\to L$.
\par\smallskip
\begin{assertion}\label{ass1}
Let $R$ be large enough.
Then, for some $S_0$,
for any $S> S_0$,
the strong density $d_n(\E^AX\cap P_{R,S})$ exists and is equal to
\begin{equation}\label{eqdwLoc}
  \frac{n!}{( 2\pi)^n} w(P)\cos(L_P^\bot,i L_P)\:A(P),
\end{equation}
where $w(P)$ is the weight of the cone $P$ (see Definition {\rm\ref{dfEAweight}}),
$\cos(L_P^\bot,i L_P)$ is the cosine of the angle between the $n$-dimensional spaces $L_P^\bot$ and $i L_P$ (i.e. the area distortion factor
under the orthogonal projection $L_P^\bot\to i L_P$),
and $A(P)$ is the angle of the $n$-dimensional cone $P$ (the full $n$-dimensional angle is considered to be one).
\end{assertion}
\begin{lemma}\label{lmass1_1}
Let $P = (K + \im\C^N) \cap L$, where $K \in \mathcal{K}_{A-P}$.
Define $\Z_P = L_P \cap \exp^{-1}\T_K$.
Then

(i)
$$
\Z_P = (\Z_K + V_K + iV_K) \cap L = (\Z_K + V_K + iV_K) \cap L_P,
$$
where $\Z_K$ is a lattice from {\rm\hyperlink{p3}{($\bf p_3$)}}.

(ii) $\Z_P$ is an $n$-dimensional lattice in the space $L_P$.
\end{lemma}

\begin{proof}
By construction,
$L_P = L \cap (V_K + \im\C^N)$,
and $\im\C^N = iV_K + iV_\mathbb{N}$.
From \hyperlink{p3}{($\bf p_3$)}, it follows that 1) $L_P \subset V_K + iV_K + iV_\mathbb{N}$,
and 2) $\Z_K \subset iV_\mathbb{N}$.
Therefore,
$$
\forall z \in \Z_K \colon\, z + (V_K + iV_K) \subset V_K + iV_K + iV_\mathbb{N},
$$
i.e., $L_P$ and $z + V_K + iV_K$ are complementary subspaces
in the space $V_K + iV_K + iV_\mathbb{N}$.
From the complex non-degeneracy of the cone $P$, it follows that $L \cap (V_K + iV_K) = 0$;
see Corollary \ref{corA-P}.
Therefore, for any $z \in \Z_K$, the intersection $L_P \cap \left(z + (V_K + iV_K)\right)$
is a single point and coincides with $L \cap \left(z + (V_K + iV_K)\right)$.
From this, the necessary statements follow.
\end{proof}
Recall that in \hyperlink{p3}{($\bf p_3$)}, the characters $x_1, \ldots, x_N$
of the torus $(\C \setminus 0)^N$ are defined.
Denote by $\Pi(x_1, \ldots, x_n)$
the $n$-dimensional parallelepiped in the space $V_K^\bot$
with the sides $x_1, \ldots, x_n$.
Let $\pi\colon {\C^N}^* \to L^*$ be a linear operator
adjoint to 
$d\exp\colon L \to \C^N$ at the zero point in the space $L$.
We denote by $\pi_K\colon V_K^\bot \to L^*$ the restriction
of $\pi$ onto the space $V_K^\bot$,
and set $\lambda_i=\pi_K(x_i)$.
Let $\Pi(\lambda_1, \ldots, \lambda_n)$ be
the $n$-dimensional parallelepiped in $L$ with the sides $\lambda_1, \ldots, \lambda_n$.
Using the Hermitian metric in $L$, consider the linear functionals $\lambda_i$
as vectors in $L$.
The following statement
directly follows from the given definitions.
\begin{corollary}\label{corass1_0}
{\rm(1)}
The vectors $\lambda_i \in L$ lie in the subspace $L_P^\bot$,
where $L_P^\bot$ is the orthogonal complement of the subspace $L_P$ in $L$ with respect to the pairing $\re\langle *, *\rangle$.

{\rm(2)}
The parallelepiped $\Pi(\lambda_1, \ldots, \lambda_n)$ lies in the subspace $L_P^\bot$ and is the image of the parallelepiped $\Pi(x_1, \ldots, x_n)$
under the described above projection $\pi_K\colon V_K^\bot \to L$.

{\rm(3)}
The lattice $\Z_P$ is an \EA-variety in $L$,
defined by the system of equations
$\E^{\langle \lambda_1, z \rangle} = \ldots = \E^{\langle \lambda_n, z \rangle} = 1$.
\end{corollary}
\begin{lemma}\label{lmass1_2}
For the strong $n$-density of the lattice $\Z_P$, it is true that
$$
d_n(\Z_P) = \frac{\cos(L_P^\bot, i L_P) \: \vol_n\left(\Pi(\lambda_1, \ldots, \lambda_n)\right)}{(2\pi)^n}
$$
\end{lemma}

\begin{proof}
From the complex non-degeneracy of the cone $P$, it follows that
the subspace $L_P$ in $L$ does not contain any proper complex subspace;
see definition {\rm\ref{dfEAFan}}.
Therefore, the desired equality
follows from Proposition \ref{prprelim12}.
\end{proof}
In what follows we use the concept of toric tongue
$t_{K,\tau}$
with base $K$ and shift $\tau$
and other notations from \textsection\ref{Tongue}.
Let $x \in \C^N$ be any of the preimages of
$\tau \in (\C \setminus 0)^N$
under the exponential map $\exp\colon \C^N \to (\C \setminus 0)^N$.
Then, according to definition \ref{dfTonque}, the exponential preimage of the toric tongue $t_{K,\tau}$
can be written, using the notations introduced in \hyperlink{p3}{($\bf p_3$)}, as
$$
  \exp^{-1}(t_{K,\tau}) = x + \Z_K + K + iV_K.
$$
Recall that, according to definition \ref{dfApproxSet}, the sets of $\varepsilon$-approximating tongues $T_\varepsilon(\E^AY)$, $T_\varepsilon(Y)$
can be chosen such that the tongues from the first set differ from the tongues from the second set by a factor of $\E^A$.
That is, if $t_{K,\tau} \in T_\varepsilon(\E^AY)$ and $\xi = \E^{-A}\tau$,
then $t_{K,\xi\tau} \in T_\varepsilon(Y)$.

\begin{corollary}\label{corass1_1}
For any tongue $t_{K,\tau} \in T_\varepsilon(\E^AY)$ with the base $K \in \mathcal{K}_{A-P}$, it holds that, for some $x \in L$,
\begin{equation}\label{eqass_1_1}
 L \cap \exp^{-1}(t_{K,\tau}) = x + \Z_P
\end{equation}
\end{corollary}
\begin{proof}
Follows from Lemma \ref{lmass1_1}.
\end{proof}
For the tongue $t_{K,\tau}$, its weight $m(t_{K,\tau})$ is defined;
see definition \ref{dfApproxSet}.
According to Corollary \ref{corass1_1},
the set $L \cap \exp^{-1}(t_{K,\tau})$ is a shift $x + \Z_P$ of the lattice $\Z_P$.
Assign to this shifted lattice the multiplicity $m(L \cap \exp^{-1}(t_{K,\tau}))$ equal to the weight $m(t_{K,\tau})$ of the tongues $t_{K,\tau}$.

Recall that the weight $w(P)$ of the cone $P \in \mathcal{E}$
is an element of degree $n$ of the ring $\mathfrak{S}_\vol(L_P^\bot)$;
see definition \ref{dfWeightedFan}.
Since a Hermitian metric is fixed in the space $L$,
this weight is characterized by a numerical value
(see Remark \ref{rmOmegaNumber}),
for which we retain the notation $w(P)$.
\begin{proposition}\label{prass1_3}
Let
$\bigcup_{\tau,K \in \mathcal{K}_{A-P}} \left(L \cap \exp^{-1}(t_{K,\tau})\right)$ be the union of shifted lattices $\Z_P$ from {\rm(\ref{eqass_1_1})}
with the multiplicities $m(L \cap \exp^{-1}(t_{K,\tau}))$ over all toric tongues from the set $T_\varepsilon(\E^AY)$ with bases $K \in \mathcal{K}_{A-P}$.
Then
$$
d_n\left(\cup_{\tau,K_{A-P}} \left(L \cap \exp^{-1}(t_{K,\tau})\right)\right) =
\frac{n!}{(2\pi)^n}\omega(P)\cos(L_P^\bot, i L_P)
$$
\end{proposition}
\begin{proof}
Let $\hat\T$ be the completion of the torus $(\C\setminus0)^N$,
corresponding to the fan $\mathcal K_{\rm trop}(Y)$.
Denote by $\hat Y_K$ the set of limit points of $Y$
on the $n$-dimensional toric orbit of the variety $\hat\T$,
corresponding to the $(A-P)$-cone $K$.
According to the property \hyperlink{tr_4}{($\bf tr_4$)} from \textsection\ref{trop},
 $\#(\hat Y_K)=n!w(K)$,
where $w(K)$ is the weight of the cone $K$.
Since the number of shifted tongues with base $K$ is equal to $\#(\hat Y_K)$,
then, using lemma \ref{lmass1_2}, we get that $d_n\left(\cup_{\tau,K} \left(L\cap \exp^{-1}(t_{K,\tau})\right)\right)$
equals
$$
\frac{n!\cos(L_P^\bot,i L_P)}{(2\pi)^n}\sum_{K\in\mathcal K_{A-P}}w(K)\vol_n(\Pi(\lambda_1,\ldots,\lambda_n))
$$
Applying corollary \ref{corass1_0} (2) and definition \ref{dfEAweight},
we get that $$\sum_{K\in\mathcal K_{A-P}}w(K)\vol_n(\Pi(\lambda_1,\ldots,\lambda_n))=w(P).$$
The proposition is proven.
\end{proof}

Now, to prove assertion \ref{ass1}, it remains to apply the theorem on $\varepsilon$-approximation.
Indeed, from theorem \ref{thmApproximation} it follows,
that $\E^AX\cap P_{R,S}$ is an $\varepsilon$-perturbation
of the part of the set
$$\bigcup_{\tau,K_{A-P}} L\cap \exp^{-1}(t_{K,\tau})$$
from proposition \ref{prass1_3},
located in the region $P_{R,S}$.
Hence, applying lemma \ref{lmprelim111},
we get the proof of assertion \ref{ass1}.
\par\smallskip
Summing the formulas from assertion \ref{ass1}
over all complex non-degenerate cones $P\in\mathcal E$,
we get the proof of theorem \ref{thmIsolated} (3).

\subsection{Theorem \ref{thmStrong}}\label{pr12}

Recall that in theorem \ref{thmStrong}, the \EA-variety $X$ is assumed to be
complex non-degenerate.

\subsubsection {Proof of theorem \ref{thmStrong} (1)}\label{prPr121}

If $X$ is not discrete,
then the variety $Y^L$ is a model
of the non-empty \EA-variety $X^L$.
Let $A\in B_{\mathfrak I,r}$, where $r$ is sufficiently large.
Then from the equalities
$$
(\E^AX)^L=L\cap \log\left((\E^AY)^L\right)=L\cap\log (\E^AY^L)=L\cap\left(A+\log Y^L\right)
$$
it follows that
the required statement reduces to the equality
\begin{equation}\label{eqL}
(-A+L)\cap \log Y^L=\emptyset
\end{equation}
To prove it, we use the following lemmas.

\begin{lemma}\label{lmCDX0}
Let the \EA-variety $X$ be complex non-degenerate.
Then, for any cone $K$ of codimension $n$ in $\mathcal K_{\rm trop}(Y)$,
one of the following holds:

(i) $(V_K+iV_K)\cap L=0$

(ii)
the real subspace $V_K+L+\im\C^N$
in $\C^N$ is proper.
\end{lemma}

\begin{proof}
Let $E=L\cap(V_K+\im\C^N)$.
If $V_K+L+\im\C^N=\C^N$,
then $\dim E=n$.
Then $E$ is a subspace,
generated by some cone $P\in\mathcal E$;
see \textsection\ref{tropEA}.
If, in this case,
$(V_K+iV_K)\cap L\ne0$,
then $E$ contains some complex subspace $\C^N$,
i.e., the cone $P$ is complex degenerate.
This contradicts the condition of complex non-degeneracy of $X$.
\end{proof}

\begin{lemma}\label{lmCDX1}
Let $\mathcal K_\C(Y)$ be the union of all highest-dimensional cones in $\mathcal K_{\rm trop}(Y)$,
such that $(V_K+iV_K)\cap L\ne0$,
and let $A\in B_{\mathfrak I,r}$.
Then, as $r$ grows,
the distance from the space $-A+L$ to the set $\mathcal K_\C(Y)+\im\C^N$
tends to infinity.
\end{lemma}

\begin{proof}
From lemma \ref{lmCDX0},
it follows that,
for $K\subset\mathcal K_\C(Y)$,
the subspace $V_K+L+\im\C^N$ in $\C^N$ belongs to
a finite set of proper subspaces $\mathfrak I(Y)$.
The required statement follows from the definition of the set $B_{\mathfrak I,r}$;
see definition \ref{df_I}.
\end{proof}

\begin{lemma}\label{lmCDX2}
As $r$ grows,
the distance from the space $-A+L$ to the set $\mathcal K_{\rm trop}(Y^L)+\im\C^N$
tends to infinity.
\end{lemma}

\begin{proof}
Let $M$ be a highest-dimensional cone in ${\mathcal K}_{\rm trop}(Y^L)$.
From proposition \ref{pratype3},
it follows that $\dim(L\cap (V_M+iV_M))> 0$.
Therefore, $\supp \mathcal K_{\rm trop}(Y^L)\subset \mathcal K_\C(Y)$,
where $\mathcal K_\C(Y)$ is defined above.
Now, the required statement follows from lemma \ref{lmCDX1}.
\end{proof}

According to the property \hyperlink{tr_5}{($\bf tr_5$)} from \textsection\ref{trop},
the variety $\log Y^L$
and the set $\supp \mathcal K_{\rm trop}(Y^L)+\im\C^N$
are located at a finite distance from each other.
Therefore, (\ref{eqL}) follows from lemma \ref{lmCDX2}.
Theorem \ref{thmIsolated} (1) is proven.

\subsubsection {Proof of theorem \ref{thmStrong} (2)}\label{prPr122}

The following statement is geometrically obvious.

\begin{lemma}\label{lmUnion}
Let $\mathcal E$ be a fan of cones
in the space $L$.
Then, for $R>0$, the following holds.
For any $P\in\mathcal E$, there exist $R(P)>0$ and $S(P)>0$,
such that the $R$-neighborhood $\mathcal E_R$ of the set $\supp(\mathcal E)$
is the union of pairwise disjoint sets $P_{R(P),S(P)}$,
where $P$ ranges over all cones from the fan $\mathcal E$.
\end{lemma}

From lemma \ref{lmUnion}, it follows
that theorem \ref{thmStrong} (2) is a consequence of theorem \ref{thmIsolated} (3)
and the following assertion \ref{ass2}.

Let the dimension of the cone
$\dim P$ in $\mathcal E$ be less than $n$, and
$P$ is not a face of any complex degenerate $n$-dimensional cone in $\mathcal E$.
Suppose that
$A\in B_{\mathfrak I,r}$,
where $r$ is sufficiently large.

\begin{assertion}\label{ass2}
Let $R>0$.
Then, for some $S_0$,
for any $S> S_0$,
the strong density $d_n(\E^AX\cap P_{R,S})=0$.
\end{assertion}

Consider the completion $\hat\T$ of the torus $(\C\setminus0)^N$,
corresponding to the fan $\mathcal K_{\rm trop}(Y)$.
Denote by $\bar Y$ the closure of the variety $Y$
in $\hat\T$.
According to the property \hyperlink{tr_1}{($\bf tr_1$)} in \textsection\ref{trop},
the variety $\bar Y$ is compact.

Let $B_1$ be the ball of radius $1$ centered at the origin in
the space $L$.
\begin{lemma}\label{lm_1_1}
If $d_n(\E^AX\cap P_{R,S})\neq0$,
then there exists an infinite sequence $z_i\in\E^AX\cap  P_{R,S}$,
such that

(i)
the number of intersection points $(z_i+B_1)\cap \E^AX$ tends to infinity

(ii)
the sequence $\exp(z_i)$ converges to the limit $g\in \E^A\bar Y$
\end{lemma}
\begin{proof}
Statement (i), due to the condition $\dim P<n$,
follows from the definition of strong density (see definition \ref{dfprelim11}).
Statement (ii) follows from the compactness of the variety $\E^A\bar Y$.
\end{proof}
First, suppose that $g\in(\C\setminus0)^N$.
From the closedness of the set $\E^AY$,
it follows that $g\in\E^{A}Y$.
From theorem \ref{thmStrong} (1), it follows that,
in a small neighborhood of $g$, the number of points of the set $\exp (\log g +B_1)\cap\E^{A}Y$ is finite.
On the other hand,
this set is the limit of the sequence of sets $\exp (z_i +B_1)\cap\E^AY$,
the number of points in which tends to infinity.
Thus,
if the point $g$ is contained in the finite part of the torus $\hat T$ (i.e., in the torus $(\C\setminus0)^N$),
then assertion \ref{ass2} is proven.

Let $K\in\mathcal K_{\rm trop}(Y)$ be a highest-dimensional cone in $\mathcal K_{\rm trop}(Y)$,
such that $g$ belongs to the toric orbit in $\hat\T$,
corresponding to the cone $K$.
Denote by $\T^K$
the toric orbit in $\hat\T$ corresponding to the cone $K$.
As usual, we identify the orbit $\T^K$ with the quotient torus $\T^K=(\C\setminus0)^N/\T_K$,
where $\T_K$ is the subtorus in $(\C\setminus0)^N$,
generated by the exponents of the cone $K$.
The closure $Y^K$ of the variety $Y$ on the toric orbit $\T^K$
is an algebraic variety in $\T^K$ of codimension $n$;
see property \hyperlink{tr_3}{($\bf tr_3$)} in \textsection\ref{trop}.

We will consider $\T^K$ as a torus with the Lie algebra $\mathfrak T^K=\C^N/(V_K+iV_K)$.
Let $\pi^K\colon\C^N\to\C^N/\mathfrak T^K$ be the projection map.
The tropical fan $\mathcal K_{\rm trop}(Y^K)$ of the variety $Y^K\subset\T^K$ is the
$V_K$-factorization of the fan $\mathcal K_{\rm trop}(Y)$;
see \hyperlink{tr_3}{($\bf tr_3$)} in \textsection\ref{trop}.
Recall,
this means that it is composed of the images of all cones containing $K$ from $\mathcal K_{\rm trop}(Y)$
under the projection map $\re\pi^{K}\colon\re\C^N\to\re\mathfrak T^K$.

The map
$$
\omega^K\colon L\hookrightarrow\C^N\to\mathfrak T^K\xrightarrow{\exp}\T^K
$$
is a homomorphism of groups.
From the condition of complex non-degeneracy of the cone $P$, it follows that $L\cap(V_K+iV_K)=\emptyset$.
Hence,
the map $\omega^K$ is injective.
Let $L^K=d\omega^K(L)$,
and consider $\exp\colon L^K\to\T^K$ as the standard winding map of the torus $\T^K$.
Denote by $X^K$ the \EA-variety in $L^K$ with the model $Y^K\subset\T^K$,
let $P^K=d\omega(P)$ and $A^K=\pi^K(A)$.

\begin{corollary}\label{corPfact}
When replacing the data $(Y\subset(\C\setminus0)^N, L,X,A,P)$ with the data
$(Y^K\subset\T^K, L^K,X^K,A^K,P^K)$
the conditions of assertion {\rm\ref{ass2}}
remain satisfied.
That is,

(i) The algebraic variety $Y^K$ in the torus $\T^K$ is a model of the zero-dimensional \EA-variety $X^K$

(ii)
$P^K$ is not a face of any complex degenerate $n$-dimensional cone in the tropical fan of the \EA-variety $X^K$,
and $\dim P^K<n$

(iii)
$A^K\in B_{\mathfrak I(Y^K),r}$,
where $r$ is sufficiently large
\end{corollary}
\begin{proof}
It follows from the above definitions.
\end{proof}
By construction, $g \in \mathbb{T}^K$. From Theorem \ref{thmStrong} (1) for the \EA-variety $X^K$, it follows that the intersection of the orbit $\exp(L)g$ with the variety $\E^{A_K}Y^K \subset \mathbb{T}^K$ at the point $g$ is discrete.

Considering the sequence of points $z_i \in \mathbb{C}^N$ from Lemma \ref{lm_1_1}
 in the local toric chart at the point $g$, 
 similarly to the case $g \in (\mathbb{C} \setminus 0)^n$, 
 we obtain a contradiction with the mentioned above discreteness of the intersection
$\exp(L)g\cap \E^{A_K}Y^K$.
 Assertion \ref{ass2} is proven.
\subsection{Theorem \ref{thmq1}}\label{pr13}

In the quasi-algebraic case, the tropical fan of the \EA-variety
consists of a single cone $\im \C^n$.
Therefore, the \EA-variety is complex non-degenerate,
and the statement on strong density follows from theorem \ref{thmStrong}.
Along with the structure of the tropical fan,
the geometry of the \EA-variety also
simplifies somewhat.
The detailed proof of the theorem is given in \cite{K22}.

Here,
omitting the details,
we mention a few circumstances
that actually ensure the derivation of the theorem based on the above considerations.

(1) The absence of cones of dimension $<n$ makes the density estimate from assertion \ref{ass2} unnecessary.

(2) The description of the set of shifted lattices
$$z_1(A)+\mathcal L_{i,1},\ldots,z_{N_i}(A)+\mathcal L_{i,N_i}$$
from the statement of the theorem is given in \textsection\ref{prPr113}.

(3) The shifts of these lattices continuously depend on the point $A\in B_{\mathfrak I,r}$.

(4) The set of lattices does not change
if $z$ belongs to some fixed connected component
of the domain $B_{\mathfrak I,r}$.
The statement about the connection of the set of lattices with the connected components
$\mathfrak C_1,\mathfrak C_2,\ldots,\mathfrak C_M$ is explained by
corollary \ref{cordRelFull} (4).

\end{document}